\documentclass{amsart}

\usepackage{amsmath,amssymb,amsfonts}
\usepackage{color,xcolor}
\usepackage{graphicx}
\usepackage{manfnt}
\usepackage{mathtools}
\usepackage{tikz}
\usepackage[unicode]{hyperref}
\usepackage[utf8]{inputenc}
\usepackage{bbm}
\usepackage{enumitem}

\DeclareMathOperator{\cl}{cl}

\DeclareMathOperator{\dom}{dom}
\DeclareMathOperator{\MA}{MA}
\DeclareMathOperator{\fin}{fin}
\DeclareMathOperator{\cov}{cov}
\DeclareMathSymbol{\mlq}{\mathord}{operators}{'134}
\DeclareMathSymbol{\mrq}{\mathord}{operators}{'42}
\renewcommand{\Delta}{\triangle}

\newtheorem{proposition}{Proposition}[section]
\newtheorem{lem}[proposition]{Lemma}
\newtheorem{lemma}[proposition]{Lemma}
\newtheorem{cor}[proposition]{Corollary}
\newtheorem*{claim}{Claim}
\newtheorem*{case}{Case}

\newtheorem{theorem}[proposition]{Theorem}

\newtheorem{definition}[proposition]{Definition}

\title[MAD families and pseudocompactness]{Maximal almost disjoint families and pseudocompactness of hyperspaces}

\begin{document}

\author[Guzm\'an]{O. Guzmán}
\address{Centro de Ciencias Matem\'aticas\\ Universidad Nacional Aut\'onoma de M\'exico\\ Campus Morelia\\Morelia, Michoac\'an\\ M\'exico 58089}
\curraddr{}
\email{oguzman@matmor.unam.mx}
\thanks{}

\author[Hru\v{s}\'{a}k]{M. Hru\v{s}\'{a}k}
\address{Centro de Ciencias Matem\'aticas\\ Universidad Nacional Aut\'onoma de M\'exico\\ Campus Morelia\\Morelia, Michoac\'an\\ M\'exico 58089}
\curraddr{}
\email{michael@matmor.unam.mx}
\thanks{{The research of the second author was supported  by a PAPIIT grants IN100317, IN104220 and CONACyT grant A1-S-16164.}}
\urladdr{http://www.matmor.unam.mx/~michael}

\author[Rodrigues]{V. O. Rodrigues}
\address{Institute of Mathematics and Statistics, University of São Paulo, Rua
do Matão, 1010, São Paulo, Brazil}
\curraddr{}
\email{vinior@ime.usp.br}
\thanks{The third author received support from FAPESP (grants 2017/15502-2 and 2019/01388-9. He was also a visiting scholar at the Fields Institute for Mathematical Research while this work was being made.}

\author[Todorcevic]{S. Todor\v{c}evi\'c}
\address{Department of Mathematics, University of Toronto, Toronto, Ontario M5S 2E4, Canada}
\curraddr{}
\email{stevo@math.utoronto.ca}
\thanks{The fourth author is partially supported by Grants from NSERC (455916) and CNRS (IMJ-PRG UMR7586).}

\author[Tomita]{A. H. Tomita}
\address{Institute of Mathematics and Statistics, University of São Paulo, Rua
do Matão, 1010, São Paulo, Brazil}
\curraddr{}
\email{tomita@ime.usp.br}
\thanks{The fifth author received support from FAPESP (grant 2019/19924-4)}

\date{\today}


	\maketitle
	
	\begin{abstract}
	We show that all maximal almost disjoint families have pseudocompact Vietoris hyperspace if and only if $\mathsf{MA}_\mathfrak c (\mathcal P(\omega)/\mathrm{fin})$ holds. We further study the question
	whether there is a maximal almost disjoint family whose hyperspace is pseudocompact and prove that consistently such families do not exist \emph{genericaly}, by constructing a consistent example of a maximal almost disjoint family $\mathcal A$ of size less than $\mathfrak c$ whose hyperspace is not pseudocompact.
	\end{abstract}

\section{Introduction and notation}
Recall that an infinite collection $\mathcal A\subseteq [\omega]^\omega$ is \emph{almost disjoint (AD)} if any two of its members have finite intersection. An AD family is \emph{maximal (MAD)} if it is not properly contained in any other almost disjoint family. 

Given an almost disjoint family $\mathcal A$, the \emph{Mrówka–Isbell} space $\Psi(\mathcal A)$ associated to $\mathcal A$ is the space $\omega\cup \mathcal A$, where $\omega$ is open and discrete and a open neighborhood basis for $A \in \mathcal A$ is $\{\{A\}\cup (A\setminus F): F \in [\omega]^{<\omega}\}$. It is straightforward to verify that this is a Hausdorff, locally compact, first countable, non compact, zero dimensional topological space, and it is \emph{pseudocompact} (every $\mathbb R$-valued continuous function on $X$ is bounded), if and only if $\mathcal A$ is maximal (see e.g.\cite{HRUSAK20073048}).

The \emph{Vietoris hyperspace} of a topological space $X$ is the set %
$$\exp(X)=\{F \subseteq X: F\neq \emptyset\, \text{and}\, F\,\text{is closed}\}$$ 
endowed with the topology generated by the sets 
$$U^-=\{F \in \exp(X): F \cap U\neq \emptyset\}\text{ and}$$ 
$$U^+=\{F \in \exp(X): F\subseteq U\},$$
where $U\subseteq X$ is open.

In \cite{ginsburg_1975}, J. Ginsburg proved that for a Tychonoff space $X$, if $\exp(X)$ is pseudocompact, then every finite power of $X$ is also pseudocompact. He asked whether there is a relation between the pseudocompactness of $X^\omega$ and that of $\exp(X)$, and asked whether it is possible to characterize those spaces which have pseudocompact hyperspace. 

J. Cao, T. Nogura and A. Tomita \cite{CAO2004133} provided a partial answer by showing that for every homogeneous Tychonoff  space $X$, if $\exp(X)$ is pseudocompact, then $X^\omega$ is pseudocompact.
On the other hand, M. Hru\v{s}\'ak, F. Hern\'andez-Hern\'andez and I. Martínez-Ruiz \cite{HRUSAK20073048} showed that, in {\sf ZFC}, there is a subspace of $\beta \omega$ containing $\omega$ such that $X^\omega$ is pseudocompact but $\exp(X)$ is not. This was extended
by V. Rodrigues, A. Tomita and Y. Ortiz-Castillo \cite{ORTIZCASTILLO20189}, who showed that there is a space $X$ such that $X^\kappa$ is countably compact for every $\kappa<\mathfrak h$, but $\exp(X)$ is still not pseudocompact. They also showed that for spaces of this kind, if $\exp(X)$ is pseudocompact, so are $\exp(X)^\omega$ and $X^\omega$.

J. Cao and T. Nogura, in a private conversation, asked whether $\exp(X)$ is pseudocompact for some/every Mrówka–Isbell space $X$. 
The first relevant observation is:

\begin{proposition}[\cite{HRUSAK20073048}]Let $\mathcal A$ be an AD family. Then  $\Psi(\mathcal A)$ is pseudocompact iff $\Psi(\mathcal A)^\omega$ is pseudcompact iff $\mathcal A$ is MAD.
\end{proposition}

In particular,  if $exp(X)$ is a pseudocompact hyperspace of a Mrówka-Isbell space then so is $X^\omega$. So Ginsburg's questions restricted to the class of Mrówka-Isbell spaces becomes the problem of characterizing those MAD families whose Mrówka-Isbell spaces are pseudocompact. 

Recall that a family  $\mathcal P\subseteq [\omega]^\omega$ is \emph{centered} if any intersection of a finite number of members of $\mathcal P$ is infinite. A set $A \in [\omega]^\omega$ is a \emph{pseudointersection} of $\mathcal P$ if $A\subseteq^*P$ (i.e. $A\setminus P$ is finite) for every $P \in \mathcal P$.  The \emph{pseudointersection number} $\mathfrak p$ is the smallest cardinality of a centered $\mathcal A\subseteq [\omega]^\omega$ with no pseudointersection. A collection $\mathcal D\subseteq[\omega]^\omega$ is \emph{open dense} if $\forall A \in [\omega]^\omega\, \exists B \in \mathcal D\, B\subseteq A$, and $\forall A \in [\omega]^\omega\, \forall B \in \mathcal D\, A\subseteq^* B\implies A\in \mathcal A$. The \emph{distributivity number} $\mathfrak h$ is the least cardinality of a family of open dense subsets of $[\omega]^\omega$ with empty intersection.

The main result of  \cite{HRUSAK20073048} states:

\begin{theorem}[\cite{HRUSAK20073048}]\label{PequalCHequalC}
\begin{enumerate}
    \item If $\mathfrak p=\mathfrak c$, then $\exp(\Psi(\mathcal A))$ is pseudocompact for every MAD family $\mathcal A$.
    \item If $\mathfrak h<\mathfrak c$, there is a MAD family $\mathcal A$ such that $\exp(\Psi(\mathcal A))$ is not pseudocompact.
\end{enumerate}
\end{theorem}

Part (2) of the theorem depends heavily on the \emph{base tree theorem} of Balcar, Pelant and Simon \cite{BalcarPelantSimon1980}
which affirms the existence of a tree $\mathcal T\subseteq [\omega]^\omega$  of height $\mathfrak h$ ordered by $\supseteq^*$, such that every element has $\mathfrak c$-many immediate successors, each level is a MAD family and such that every infinite subset of $\omega$ has a subset in the tree.

In \cite{Rodrigues2019}, V. Rodrigues and A. Tomita showed that after adding $\omega_1$ Cohen Reals there is a Cohen indestructible MAD family of cardinality $\omega_1$ whose Mrówka-Isbell space has pseudocompact hyperspace.

\medskip

In this article we optimize the above theorem by showing (Theorem \ref{char}) that the statement that all MAD families have pseudocompact Vietoris hyperspace is equivalent to  the assertion $\mathsf{MA}_\mathfrak c (\mathcal P(\omega)/\mathrm{fin})$ holds\footnote{If $\kappa$ is a cardinal and $P$ is a pre-order, $\MA_{\kappa}(P)$ is the statement ``for every collection of $\leq \kappa$ dense subsets of $P$ there exists a filter $G$ on $P$ which intersects every dense set of the collection''. The boolean algebra
$\mathcal P(\omega)/\fin$ can be seen as the set $[\omega]^\omega$ ordered by $\subseteq^*$.}.

The problem of whether there is a pseudocompact hyperspace of a Mrówka-Isbell space in {\sf ZFC} was raised in \cite{HRUSAK20073048} and is still open. Here we provide a partial answer to the problem by showing that it is consistent that there is
a MAD family $\mathcal A$ of size strictly less than $\mathfrak c$
such that $\exp(\Psi(\mathcal A))$ is not pseudocompact. In particular, it shows that it is consistent that MAD families with pseudocompact hyperspaces do not exist generically \footnote{Recall \cite{guzmanetal} A MAD family with property $\varphi$ exists \emph{generically} if any AD family of size less than $\mathfrak c$ can be extended to a MAD family satisfying $\varphi$.}.

\medskip

Our notation is mostly standard. In particular, $\omega$ denotes the set of finite von Neumann ordinals and is identified with the natural numbers. The set of free ultrafilters over $\omega$ is denoted by $\omega^*$ and is identified with the remainder of the Stone-\v Cech compactification of $\omega$. Given $p \in \omega^*$, a topological space $X$, $x \in $ and a sequence $f:\omega\rightarrow X$, we say that $x$ is a \emph{$p$-limit} of $f$ if for every neighborhood $U$ of $x$, the set $\{n \in \omega: f(n) \in U\}$ belongs to $p$ and we write $p$-$\lim f=x$. 

The smallest cardinality of a MAD family is defined as $\mathfrak a$. It is well known that $\omega_1\leq \mathfrak p\leq \mathfrak h\leq \mathfrak a\leq \mathfrak c$ and that all inequalities are consistently strict. See \cite{BlassCombinatorics} for more on cardinal invariants of the continuum.

\section{Equivalence with \texorpdfstring{$\MA_{\mathfrak c}(\mathcal P(\omega)/\fin)$}{MAcPw}}

In this section we shall identify statements equivalent  to the assertion ``For every MAD family $\mathcal A$, $\exp(\Psi(\mathcal A))$ is pseudocompact''.

The following proposition appears as Proposition 2.1 in \cite{Rodrigues2019}:

\begin{proposition}\label{NeccessaryAndSufficientGood}Let $\mathcal A$ be an almost disjoint family. Then $\exp(\Psi(\mathcal A))$ is pseudocompact if and only if every sequence $F:\omega\rightarrow [\omega]^{<\omega}\setminus\{\emptyset\}\subseteq \exp(\Psi(\mathcal A))$ of pairwise disjoint sets has an accumulation point.
\end{proposition}

By restricting ourselves to pairwise disjoint sequences, we can get a result similar to one found in \cite{HRUSAK20073048} which appears as Lemma 3.1.

\begin{lemma}
Let $\mathcal A$ be an almost disjoint family. Let $F=(F_n: n \in \omega)$ be a sequence of pairwise disjoint finite nonempty subsets of $\omega$. Given $A\subseteq \omega$, let $I_A=\{n \in \omega: F_n\cap A\neq \emptyset\}$ and $M_A=\{n \in \omega: F_n\subseteq A\}$. Then:
\end{lemma}

\begin{enumerate}
    \item If $L$ is a limit point of the sequence $F$ in $\exp(X)$, then $L\subseteq \mathcal A$, and
    \item Given $L \subseteq \mathcal A$, $L$ is a limit point of $F$ if, and only if for every $P \subseteq \omega$ such that $\forall A \in L \, A\subseteq P$, the set $\{I_A: A \in L\}\cup\{M_P\}$ is centered.
\end{enumerate}

\begin{proof}
For the first item, notice that if $n \in \omega\cap L$, then $\{n\}^{-}$ is a neighborhood of $L$ which intersects at most one element from the sequence $F$, so $L$ cannot be a limit point for $F$.

For the second item, first suppose $L$ is a limit point for $F$. Fix arbitrary $A_0, \dots, A_l \in L$ and $P$ as in the item, we must show that $I_{A_0}\cap\dots\cap I_{A_n}\cap M_P$ is infinite. Fix $k \in \omega$. Notice that $L\cup (P\setminus k)$ is open, so $V=(L\cup P)^{+}\cap (\{A_0\}\cup A_0)^-\dots\cap (\{A_n\}\cup A_l)^-$ is a nhood of $L$, so it must have a point $F_n$ with $n\geq k$. Then $F_n\subseteq P$ and $F_n\cap A_i\neq \emptyset$ for each $i$, that is, $n \in I_{A_0}\cap\dots\cap I_{A_l}\cap M_P\setminus k$. Since $k$ is arbitrary we are done.

Now we prove the converse. Let $U_0, \dots, U_n, V$ be open sets of $\Psi(\mathcal A)$ such that $L \in U_0^-\cap \dots \cap U_n^-\cap V^+$. Let $P=V\cap \omega$ and, for each $i\leq l$, let $A_i \in L\cap U_i$ and let $k_i$ be such that $A_i\setminus k_i\subseteq U_i$. Then $I_{A_0}\cap \dots \cap I_{A_l}\cap M_P$ is infinite. Since $F$ is a pairwise disjoint sequence, there exists $m$ such that for all $n\geq m$, $F_n\cap \max\{k_0, \dots, k_l\}=\emptyset$. Let $m\geq n$ be in $I_{A_0}\cap \dots \cap I_{A_l}\cap M_P$. Then $F_m \in U_0^-\cap \dots \cap U_l^-\cap V^+$ and the proof is complete.
\end{proof}

A sufficient condition to guarantee the existence of a limit point is the following lemma:

\begin{lem}\label{SufficientPLimit} Let $\mathcal A$ be an almost disjoint family, $p$ be a free ultrafilter and $F:\omega\rightarrow [\omega]^{<\omega}\setminus\{\emptyset\}\subseteq \exp(\Psi(\mathcal A))$ be a sequence of pairwise disjoint sets. Then if for every sequence $f \in \prod_{n \in \omega}F_n$ there exists $A \in \mathcal A$ and $B \in p$ such that $f[B]\subseteq A$, it follows that $F$ has a $p$-limit.
\end{lem}

\begin{proof}Let $P=\prod_{n \in \omega} F_n$.	Given $f \in P$, fix $B_f \in p$ and $A_f \in \mathcal A$ such that $f[B_f]\subseteq A_f$. Let $\mathcal B=\{A_f: f \in P\}$ We claim that $p$-$\lim F=\mathcal B$.
	
	To verify the claim, it suffices to verify the $p$-limit condition for sub-basic sets, so let $U\subset \Psi(\mathcal A)$ be open.
	
	If $\mathcal B \in U^-$, then there exists $f \in P$ with $A_f\in U$. Since $U$ is open, $A_f\subseteq^* U$. Then $f[B_f]\subseteq^* U$. So $B_f\subseteq^*\{n \in \omega: f(n) \in U\}\subseteq \{n \in \omega: F_n\in U^-\}$. Since $B_f \in p$ and $p$ is a free ultrafilter, it follows that $\{n \in \omega: F_n\in U^-\} \in p$. 
	
	If $\mathcal B \in U^+$, suppose by contradiction that $\{n \in \omega: F_n \in U^+\}\not \in p$. Then $I=\{n \in \omega:F_n\setminus U\neq \emptyset\}\in p$. Let $f \in P$ be such that for each $n \in I$, $f(n) \in F_n\setminus U$. Then $f[I\cap B_f]\subseteq^* A_f$ and $f[I\cap B_f]\setminus U$ is infinite, so $A_f\setminus U$ is infinite. On the other hand, since $\mathcal B \in U^+$ we have $A_f \in U$, but $U$ is open, so $A_f\subseteq^* U$, a contradiction.
\end{proof}

Given a $T_1$ topological space $X$ with no isolated points, the \emph{Baire number} of $X$, denoted by $n(X)$, is the smallest cardinality of a family of open dense subsets of $X$ with empty intersection. In the following theorem, the equivalence between a) and d) with an arbitrary infinite $\kappa$ in the place of $\mathfrak c$ was presented without proof in \cite{BalcarPelantSimon1980}. For the sake of completeness, we present a proof (in the proof we present, one could switch $\mathfrak c$ for any other infinite cardinal).

	\begin{theorem}\label{char} The following are equivalent:
	\begin{enumerate}[label=\alph*)]
	\item $\MA_{\mathfrak c}(\mathcal P(\omega)/\fin)$
	\item For every MAD family $\mathcal A$, $\exp(\Psi(\mathcal A))$ is pseudocompact,
	\item $\mathfrak h=\mathfrak c$ and every base tree has a cofinal branch
	\item $n(\omega^*)>\mathfrak c$.
	\end{enumerate}
	\end{theorem}

	\begin{proof}
	
	$a)\rightarrow b)$ Suppose $\MA_{\mathfrak c}(\mathcal P(\omega)/\fin)$ and fix a MAD family $\mathcal A$. Let $F:\omega\rightarrow[\omega]^{<\omega}\setminus\{\emptyset\}$ be a pairwise disjoint sequence. Let $P=\prod_{n\in \omega}F_n$. Given $f \in P$, let 
	$$D_f=\{B \in [\omega]^\omega:\exists A \in \mathcal A\, f[B]\subseteq A\}.$$
	It is straightforward to verify $D_f$ is is dense in $\mathcal P(\omega)/\fin$. By $\MA_{\mathfrak c}(\mathcal P(\omega)/\fin)$, let $p$ be a filter intersecting every member of $\{D_f: f \in P\}$. Then, by Lemma \ref{SufficientPLimit}, $F$ has a $p$-limit. Now the conclusion follows from Proposition \ref{NeccessaryAndSufficientGood}.
	
	$b)\rightarrow c)$	Negating $c)$, either $\mathfrak h<\mathfrak c$ or there exists a base tree of height $\mathfrak c$ with no cofinal branches. Either way, there is a base tree with no branches of cardinality $\mathfrak c$, so the negation of $b)$ follows from the second statement of Proposition \ref{PequalCHequalC}.
	
	$c)\rightarrow a)$ Suppose $\neg\MA_{\mathfrak c}(\mathcal P(\omega)/\fin)$ and $\mathfrak h=\mathfrak c$. By $\neg\MA_{\mathfrak c}(\mathcal P(\omega)/\fin)$, there exists a family $(\mathcal A_\alpha: \alpha<\mathfrak c)$ of MAD families such that for every $p \in \omega^*$, $p\cap A_\alpha=\emptyset$ for some $\alpha<\mathfrak c$. Using $\mathfrak h=\mathfrak c$ and following the standard construction of a base tree (e.g. \cite{BlassCombinatorics}), there exists a base tree $\mathcal T$ of height $\mathfrak c$ such that every level $\mathcal T_\alpha$ of $\mathcal T$ refines every element of $\{\mathcal A_\beta: \beta<\alpha\}$ (that is: given $\beta<\alpha$ and $A \in \mathcal T_\alpha$, there exists $B \in \mathcal A_\beta$ such that $A\subseteq^*B$). Then $\mathcal T$ cannot have a cofinal branch, for if it had, we would be able to extend it to an ultrafilter, and this ultrafilter would intersect every $\mathcal T_\alpha$ for every $\alpha<\mathfrak c$, so it would also intersect $\mathcal A_\alpha$ for every $\alpha<\mathfrak c$.
	
	 $a)\rightarrow d)$ suppose $\MA_{\mathfrak c}(\mathcal P(\omega)/\fin)$ and let $(U_\alpha: \alpha<\mathfrak c)$ be a collection of open dense subsets of $\omega^*$. For each $\alpha$, let $\mathcal A_\alpha$ be an infinite almost disjoint family such that $A^*\subseteq U_\alpha$ for every $A \in \mathcal A_\alpha$ maximal for this property. It is easy to verify that each $\mathcal A_\alpha$ is a MAD family. Since MAD families are maximal antichains on $\mathcal P(\omega)/\fin$, by $\MA_{\mathfrak c}(\mathcal P(\omega)/\fin)$ there exists an filter that intersects all of them. We can extend this filter to an ultrafilter $p$. For every $\alpha<\omega$, there exists $A \in A_\alpha\cap p$, so $p \in A^*\subseteq U_\alpha$, therefore $p \in \bigcap_{\alpha<\mathfrak c}U_\alpha$.
	 
	 $d)\rightarrow a)$ Suppose $n(\omega^*)>\mathfrak c$ and let $(\mathcal B_\alpha: \alpha<\omega_1)$ dense subsets of $\mathcal P(\omega)/\fin$. For each $\alpha<\mathfrak c$, let $U_\alpha=\bigcup\{B^*: B \in \mathcal B_\alpha\}$. It is easy to verify $U_\alpha$ is open and dense. Let $p \in \bigcap_{\alpha<\mathfrak c} U_\alpha$. Then for each $\alpha<\mathfrak c$ there exists $B \in \mathcal B_\alpha$ such that $p \in B^*$, that is, $B \in p\cap \mathcal B_\alpha$.
	\end{proof}

    Next we present a model of $\mathfrak p<\mathfrak c$ where all Mrówka-Isbell spaces from MAD families have pseudocompact hyperspaces.

	\begin{theorem} It is consistent that $\mathfrak p<\mathfrak c$ and $exp(\mathcal A)$ is pseudocompact for every  MAD family $\mathcal A$. 
	\end{theorem}
	
	\begin{proof}
	Suppose $V\vDash\mathfrak p=\mathfrak c=\omega_2 +\text{there exists a Suslin Tree}$. Let $S$ be a well pruned Suslin tree and let $G$ be $S$ generic over $\mathbf V.$ It is well known that $S$ forces $\mathfrak p=\omega_1<\mathfrak c$ (see, for example, \cite{Farah1996}). Suppose $\mathcal A$ is a MAD family in $V[G]$.
	
	\smallskip
	
	\textbf{Claim:} there exists a MAD family $\mathcal B \in \textbf{V}$ such that for every $B \in \mathcal B$ there exists $A \in \mathcal A$ such that $B\subseteq^*A$.

\begin{proof}[Proof of the claim]

Let $\mathring{\mathcal A}$ be a name for $\mathcal A$ and let $p\in S$ be such that $p\Vdash \mathring{\mathcal A}$ is a MAD family. If $t\leq p$, let $\mathcal A_t=\{A \in [\omega]^\omega: t\Vdash \check A\in \mathring{\mathcal A}\}$. Each of there sets is an almost disjoint family. In $\mathbf V$, for each $t\leq p$ let $\mathcal B_t$ be a MAD family containing $\mathcal A_t$.

Since $|S|=\omega_1<\mathfrak h$, there exists $\mathcal B$ refining $\{\mathcal B_t: t\leq p\}$, that is, for every $B \in \mathcal B$ and for every $t\leq p$, there exists $A \in \mathcal B_t$ such that $B\subseteq^*A$.

We show $\mathcal B$ is as intended: given $B \in \mathcal B$, there exists $A \in \mathcal A$ such that $|B\cap A|=\omega$. Since forcing with Suslin trees do not add reals, there exists $t\leq p$ such that $t\Vdash A\in \mathring{\mathcal A}$, so $A \in \mathcal A_t$. There exists $A'\in \mathcal B_t$ such that $B\subseteq^* A'$. Since $A', A \in \mathcal B_t$, it follows that $A=A'$, which completes the proof of the claim.
\end{proof}

 Let $F \in \mathbf V[G]$ be a sequence of pairwise disjoint finite nonempty subsets of $\omega$. Since forcing with $S$ does not add reals, $F \in \mathbf{V}$. Working in $\mathbf V$, since $\mathfrak p=\mathfrak c$ holds, there exists a free ultrafilter $p$ for which every $f \in \prod_{n \in \omega}F_n$ there is $I \in p$ such that $f[I]$ is contained in an element of $\mathcal B$.
 
 In $\mathbf V[G]$, $p$ is still a free ultrafilter and for every $f \in \prod_{n \in \omega}F_n$ there is $I \in p$ such that $f[I]$ is contained in an element of $\mathcal A$. This implies that every such an $f$ has a $p$-limit in $\Psi(\mathcal A)$ and that in the hyperspace, $p$-$\lim F=\{p\text{-}\lim f: f \in \prod_{n \in \omega} F_n\}$.
	\end{proof}

	\section{Generic existence of pseudocompact MAD families}
	
	In this section we study sufficient conditions for the existence of MAD families with pseudocompact hyperspaces which we shall call \emph{pseudocompact}. In particular we give sufficient conditions for the existence of both large and small pseudocompact MAD families. Following \cite{guzmanetal} we shall say that \emph{pseudocompact MAD families exist generically} if every AD family of size less than $\mathfrak c$ can be extended to a pseudocompact one. Of course, it follows from the results of the previous section that pseudocompact MAD families exist generically if the conditions of Theorem \ref{char} are satisfied, i.e. if $\mathfrak h=\mathfrak c$ and every base tree has a cofinal branch.
	
	On the other hand, this is not equivalent to the generic existence of pseudocompact MAD families which we shall show next. Recall \cite{brendle-shelah} that 
	given an ultrafilter $\mathcal U$ the \emph{pseudointersection number} $\mathfrak p(\mathcal U)$ of $\mathcal U$ is defined as the minimal size of a subfamily $\mathcal X$ of $\mathcal U$ without  a pseudointersection in $\mathcal U$. I.e. $\mathfrak p(\mathcal U)>\omega$ if and only if $\mathcal U$ is a \emph{$P$-point}, and $\mathfrak p(\mathcal U)=\mathfrak c$ if and only if $\mathcal U$
	is a \emph{simple $P_\mathfrak c$-point} i.e. an ultrafilter generated by a $\subseteq^*$-decreasing chain of length $\mathfrak c$.

	\begin{theorem} If $\mathcal A$ is a MAD family, $\mathcal U$ an ultrafilter and $|\mathcal A|< \mathfrak p (\mathcal U)$ then $\exp(\Psi(\mathcal A))$ is pseudocompact.
	\end{theorem}
	
	\begin{proof} Let $\mathcal U$ be given. Fix a MAD family $\mathcal A$ such that $|\mathcal A|<\mathfrak p(\mathcal U)$. By Lemma \ref{SufficientPLimit}, and Proposition \ref{NeccessaryAndSufficientGood}, it is sufficient to verify that for every injective sequence $f:\omega\rightarrow \omega$ there exists $B \in U$ and $A \in \mathcal A$ such that $f[B]\subseteq A$.
	
	Suppose this is not the case. Then there exists $f:\omega\rightarrow \omega$ such that for all $A \in \mathcal A$ and $B \in \mathcal U$, $f[B]\setminus A$ is infinite. 
	
	First, notice that given $A \in \mathcal A$, there exists $B_A \in \mathcal U$ such that $f[B_A]\cap A$ is empty: the sets $\{n \in \omega: f(n)\notin A\}$ and $\{n \in \omega: f(n)\in A\}$ form a partition of $\omega$, so one of them is in $\mathcal U$. But the second is not in $\mathcal U$ by hypothesis. Let $B_A$ be the first set.
	
	Now let $B$ be a pseudointersection of $\{B_A: A \in \mathcal A\}$ in $\mathcal U$. It follows that $f[B]\cap A$ is finite for every $A \in \mathcal A$, contradicting the maximality of $\mathcal A$.
	\end{proof}
	
	Note that the same argument shows that:
	
	\begin{cor} If there is an ultrafilter $\mathcal U$ such that $\mathfrak p(\mathcal U)=\mathfrak c$ then pseudocompact MAD families exist generically.
	\end{cor}
	
Next we will construct a model where this actually happens, i.e. $\mathfrak{a}=\omega_{1}$ and there is an ultrafilter $\mathcal{U}$ such that $\mathfrak{p}\left(  \mathcal{U}\right)
=\omega_{2}.$ We will use the method of matrix iterations, which was
introduced by Blass and Shelah in \cite{UltrafilterswithSmallGeneratingSets}
and further developed by Brendle and Fischer in
\cite{MadFamiliesSplittingFamiliesandLargeContinuum}. We will provide a quick
review of this method, but it would be helpful if the reader had familiarity
with \cite{MadFamiliesSplittingFamiliesandLargeContinuum}. To learn more about
matrix iterations, the reader may consult
\cite{MatrixIterationsandCichonDiagram, SplittingBoundingandAlmostDisjointnesscanbequitedifferent,
CoherentSystemsofFiniteSupportIterations,
VanDouwenDiagramforDenseSetsofRationals,
PseudoPpointsandSplittingNumber, OntheCofinalityoftheSplittingNumber}.

\smallskip

The following forcing was introduced by Hechler \cite{Hechler} for adding generically a MAD family:
Let $\gamma\geq\omega_{1}.$ Define $\mathbb{H}_{\gamma}$ as the set of all
functions $p$ such that there are $F_{p}\in\left[  \gamma\right]  ^{<\omega}$
and $n_{p}\in\omega$ such that $p:F_{p}\times n_{p}\longrightarrow2.$

It also appears in \cite{MadFamiliesSplittingFamiliesandLargeContinuum}.
Given $p,q\in\mathbb{H}_{\gamma},$ define $p\leq q$ if the following holds:

\begin{enumerate}
\item $q\subseteq p$ (hence $F_{q}\subseteq F_{p}$ and $n_{q}\leq n_{p}$).

\item For every $\alpha,\beta\in F_{q}$ (with $\alpha\neq\beta$) and
$i\in\lbrack n_{q},n_{p}),$ if $p\left(  \alpha,i\right)  =1,$ then $p\left(
\beta,i\right)  =0.$
\end{enumerate}

Assume $G\subseteq\mathbb{H}_{\gamma}$ is a generic filter. For every
$\alpha<\gamma,$ define 
$$A_{\alpha}^{G}=\left\{  i\mid p\in G\left(  p\left(
\alpha,i\right)  =1\right)  \right\}  .$$
Define the \emph{generic \textsf{AD}
family }as $\mathcal{A}_{\gamma}^{G}=\left\{  A_{\alpha}^{G}\mid\alpha
<\gamma\right\}  .$ The following lemma is well known and easy to see:

\begin{lem}
Let $\gamma\leq\omega_{1}$ and $G\subseteq\mathbb{H}_{\gamma}$ a generic filter.

\begin{enumerate}
\item If $\alpha<\gamma,$ then $A_{\alpha}^{G}$ is infinite.

\item $\mathcal{A}_{\gamma}^{G}$ is an \textsf{AD} family.

\item If $\delta<\gamma,$ then $\mathbb{H}_{\delta}$ is a regular suborder of
$\mathbb{H}_{\gamma}.$

\item If $\gamma=\omega_{1},$ then $\mathcal{A}_{\omega_{1}}^{G}$ is a
\textsf{MAD} family.
\end{enumerate}
\end{lem}

More properties and preservation results may be consulted in
\cite{DiegoHechler}. 

\medskip

Let $\mathcal{F}$ be a filter on $\omega$. Define the \emph{Mathias forcing of
}$\mathcal{F}$ (denoted as $\mathbb{M(\mathcal{F})}$) \cite{mathias} as the set of all
$p=\left(  s_{p},F_{p}\right)  $ such that  $s_{p}\in\left[  \omega\right]  ^{<\omega}$ and $F_{p}\in\mathcal{F}$, ordered by 
$p=\left(  s_{p},F_{p}\right) \leq q=\left(  s_{q},F_{q}\right)$
 if  $s_{q}\subseteq s_{p}$,  $F_{p}\subseteq F_{q}$ and 
$s_{p}\setminus s_{q}\subseteq F_{q}$.

\medskip 

If $G\subseteq\mathbb{M}\left(  \mathcal{F}\right)  $ is a generic filter, the
\emph{generic real of }$\mathbb{M}\left(  \mathcal{F}\right)  $\emph{ }is
defined as 
$$r_{G}=\bigcup
\left\{  s_{p}\mid\exists p=\left(  s_{p},F_{p}\right)  \in G\right\}  .$$ 
It
is easy to see that $r_{G}$ is a pseudointersection of $\mathcal{F}.$

\smallskip

The
following notion was introduced in
\cite{MadFamiliesSplittingFamiliesandLargeContinuum}:
Let $M\subseteq N$ be transitive models of {\sf ZFC } (we may assume that $N$
is a forcing extension of $M$). Let $\mathcal{A}=\left\{  A_{\alpha}\mid
\alpha\in\gamma\right\}  $ be an AD family in $M$ and $B\in N$ an
infinite subset of $\omega.$ We say that $\bigstar_{\mathcal{A},B}^{M,N}$
holds, if 
$$\forall h:\omega\times\left[  \gamma\right]  ^{<\omega
}\longrightarrow\omega \text{ such that }  h\in M, \ \forall m\in\omega\ \text{and }\forall F\in\left[
\gamma\right]  ^{<\omega} $$ 
$$ \exists n\geq m \ [n,h\left(
n,F\right)  )\setminus\bigcup_{\alpha\in F}
A_{\alpha}\subseteq B.$$

It is easy to see that if $\bigstar_{\mathcal{A},B}^{M,N}$ hold, then
$B\in\mathcal{I}\left(  \mathcal{A}\right)  ^{+}.$

The following is easy:

\begin{lem}
Let $M\subseteq N$ be transitive models of \emph{ZFC, }$\mathcal{A}=\left\{
A_{\alpha}\mid\alpha\in\gamma\right\}  \in M$ an \textsf{AD} family and $B\in
N$ such that $\bigstar_{\mathcal{A},B}^{M,N}$ holds. If $X\in
\mathcal{I}\left(  \mathcal{A}\right)  ^{+}\cap M$ then $B\cap X$ is infinite
(in $N$). \label{Estrellita pega} \emph{ }
\end{lem}

The following is Lemma 4 of
\cite{MadFamiliesSplittingFamiliesandLargeContinuum}:

\begin{lem}
[\cite{MadFamiliesSplittingFamiliesandLargeContinuum}]Let $\gamma\leq
\omega_{1}$ and $G\subseteq\mathbb{H}_{\gamma}$ a generic filter. For every
$\xi\leq\gamma,$ define $G_{\xi}=\mathbb{H}_{\xi}\cap G.$ If $\alpha\leq
\beta<\gamma,$ then $\bigstar_{\mathcal{A}_{\alpha}^{G_{\alpha}},A_{\beta}%
}^{V\left[  G_{\alpha}\right]  ,V\left[  G_{\beta+1}\right]  }$
holds.\label{estrellita sucesor}
\end{lem}

The following is a deep result of Brendle and Fischer (crucial lemma 7 of
\cite{MadFamiliesSplittingFamiliesandLargeContinuum}):

\begin{proposition}
[Brendle, Fischer \cite{MadFamiliesSplittingFamiliesandLargeContinuum}]Let
$M\subseteq N$ be transitive models of \emph{ZFC, }$\mathcal{A}=\left\{
A_{\alpha}\mid\alpha\in\gamma\right\}  \in M$ an \textsf{AD} family and $B\in
N$ such that $\bigstar_{\mathcal{A},B}^{M,N}$ holds. Let
$\mathcal{U}\in M$ be an ultrafilter. There is an ultrafilter $\mathcal{W}\in
N$ such that the following holds:

\begin{enumerate}
\item $\mathcal{U}\subseteq\mathcal{W}$ (hence $\mathbb{M}\left(
\mathcal{U}\right)  \subseteq\mathbb{M}\left(  \mathcal{W}\right)  $).

\item If $L\subseteq\mathbb{M}\left(  \mathcal{U}\right)  $ is a maximal
antichain with $L\in M,$ then $L$ is also a maximal antichain of
$\mathbb{M}\left(  \mathcal{W}\right)  .$

\item If $G_{\mathcal{W}}\subseteq\mathbb{M}\left(  \mathcal{W}\right)  $ is
an $\left(  N,\mathbb{M}\left(  \mathcal{W}\right)  \right)  $-generic filter,
then $G_{\mathcal{U}}=G_{\mathcal{W}}\cap\mathbb{M}\left(  \mathcal{U}\right)
$ is an $\left(  M,\mathbb{M}\left(  \mathcal{U}\right)  \right)  $-generic filter.

\item $r_{G_{\mathcal{W}}}=r_{G_{\mathcal{U}}}$ (in particular,
$r_{G_{\mathcal{W}}}\in M[G_{\mathcal{U}}]$).

\item $\bigstar_{\mathcal{A},B}^{M[G_{\mathcal{U}}],\text{ }N[G_{\mathcal{W}%
}]}$ holds.
\end{enumerate}
\end{proposition}

Note that points 3 and 4 follow from points 1 and 2. It is important to note
that in general (in $N$) $\mathbb{M}\left(  \mathcal{U}\right)  $ will not be
a regular suborder of $\mathbb{M}\left(  \mathcal{W}\right)  $ (except in the
trivial case where $\mathcal{U}=\mathcal{W}$). This is because in point 2, we
only have the results for the maximal antichains that are in $M,$ but it may
fail for those that are in $N.$

\medskip

Let $\kappa$ and $\lambda$ be two cardinals. We will say that
$$\left(  \left\langle \mathbb{P}_{\alpha,\beta}\mid\alpha\leq\kappa,\beta
\leq\lambda\right\rangle ,\langle\mathbb{\dot{Q}}_{\alpha,\beta}\mid\alpha
\leq\kappa,\beta<\lambda\mathbb{\rangle}\right)  $$ is a \emph{standard matrix
iteration }if the following holds for every $\alpha\leq\kappa,\beta\leq
\lambda:$

\begin{enumerate}
\item If $\beta<\lambda,$ then$\ \mathbb{\dot{Q}}_{\alpha,\beta}$ is a
$\mathbb{P}_{\alpha,\beta}$-name for a partial order with the countable chain condition.

\item If $\beta<\lambda,$ then$\ \mathbb{P}_{\alpha,\beta+1}=\mathbb{P}%
_{\alpha,\beta}\ast\mathbb{\dot{Q}}_{\alpha,\beta}.$

\item If $\xi<\beta,$ then $\mathbb{P}_{\alpha,\xi}$ is a regular suborder of
$\mathbb{P}_{\alpha,\beta}.$

\item If $\beta$ is a limit ordinal, then $\mathbb{P}_{\alpha,\beta}$ is the
finite support iteration of $\langle\mathbb{P}_{\alpha,\xi}\mid\xi
<\beta\rangle.$

\item If $\eta<\alpha,$ then $\mathbb{P}_{\eta,\beta}$ is a regular suborder
of $\mathbb{P}_{\alpha,\beta}.$

\item If $\alpha$ is a limit ordinal, then $\mathbb{P}_{\alpha,0}$ is the
finite support iteration of $\langle\mathbb{P}_{\eta,0}\mid\eta<\alpha
\rangle.$

\item If $p\in\mathbb{P}_{\kappa,\beta},$ then there is $\gamma<\kappa$ such
that $p\in\mathbb{P}_{\gamma,\beta}.$

\item If $\dot{f}$ is a $\mathbb{P}_{\kappa,\beta}$-name for a real, then
there is $\gamma<\kappa$ such that $\dot{f}$ is a $\mathbb{P}_{\gamma,\beta}$-name.
\end{enumerate}

In the above situation, given $\alpha\leq\kappa,\beta\leq\lambda,$ we 
denote by $V_{\alpha\beta}$ the extension of $V$ by forcing with
$\mathbb{P}_{\alpha,\beta}.$

\medskip 

We now define $\left(  \left\langle \mathbb{P}_{\alpha,\beta}\mid\alpha
\leq\omega_{1},\beta\leq\omega_{2}\right\rangle ,\langle\mathbb{\dot{Q}%
}_{\alpha,\beta}\mid\alpha\leq\omega_{1},\beta<\omega_{2}\mathbb{\rangle
}\right)  $ such that for every $\alpha\leq\omega_{1}$ and $\beta\leq
\omega_{2}$ with the following properties:

\begin{enumerate}
\item $\mathbb{P}_{\alpha0}=\mathbb{H}_{\alpha}.$

\item Let $\mathcal{A}_{\alpha}=\left\{  A_{\xi}\mid\xi<\alpha\right\}  $ be
the \textsf{AD} family added by $\mathbb{H}_{\alpha}.$

\item For every $\beta<\omega_{2},$ there is a sequence $\langle
\mathcal{U}_{\gamma\beta}\mid\gamma\leq\omega_{1}\rangle$ with the following properties:

\begin{enumerate}
\item $\mathcal{U}_{\gamma\beta}\in V_{\gamma\beta}$ and it is an ultrafilter
in such model.

\item For every $\gamma<\delta\leq\omega_{1}$ the following holds:

\begin{enumerate}
\item $\mathcal{U}_{\gamma\beta}\subseteq\mathcal{U}_{\delta\beta}.$

\item If $L\subseteq\mathbb{M}\left(  \mathcal{U}_{\gamma\beta}\right)  $ is a
maximal antichain with $L\in V_{\gamma\beta},$ then $L$ is also a maximal
antichain of $\mathbb{M}(\mathcal{U}_{\delta\beta}).$

\item If $\bigstar_{\mathcal{A}_{\gamma},A_{\gamma}}^{V_{\gamma\beta
},V_{(\gamma+1)\beta}}$, and $H\subseteq\mathbb{M}\left(  \mathcal{U}%
_{(\gamma+1)\beta}\right)  $ is a   $(V_{(\gamma+1)\beta},\mathbb{M}\left(
\mathcal{U}_{(\gamma+1)\beta}\right)  )$
-generic filter, then $\bigstar
_{\mathcal{A}_{\gamma},A_{\gamma}}^{V_{\gamma\beta}[H],V_{(\gamma+1)\beta}%
[H]}$ holds. \label{paso inductivo para estrellita}
\end{enumerate}
\end{enumerate}

\item If $\beta<\omega_{2},$ then $\mathbb{P}_{\alpha,\beta}\Vdash
``\mathbb{\dot{Q}}_{\alpha,\beta}=\mathbb{\dot{M}}\left(  \mathcal{U}%
_{\alpha\beta}\right)  \mrq$ and $\mathbb{P}_{\alpha,\beta
+1}=\mathbb{P}_{\alpha,\beta}\ast\mathbb{\dot{M}}\left(  \mathcal{U}%
_{\alpha\beta}\right)  .$

\item If $\beta<\omega_{2}$ and $r_{\beta}$ is the $\left(  V_{\omega_{1}%
\beta},\mathbb{M}\left(  \mathcal{U}_{\omega_{1}\beta}\right)  \right)
$-generic real, then $r_{\beta}\in\mathcal{U}_{0\left(  \beta+1\right)  }.$

\item If $\beta<\omega_{2}$ is a limit ordinal, then $\left\{  r_{\eta}%
\mid\eta<\beta\right\}  \subseteq\mathcal{U}_{0\beta}.$
\label{ultrafiltro en limite}
\end{enumerate}

By the construction, it follows that $\left\{  r_{\beta}\mid\beta<\omega
_{2}\right\}  $ is a $\subseteq^{\ast}$-decreasing sequence (this is why point
\ref{ultrafiltro en limite} makes sense). The main point is of course, that
our 
$$\left(  \left\langle \mathbb{P}_{\alpha,\beta}\mid\alpha\leq\omega
_{1},\beta\leq\omega_{2}\right\rangle ,\langle\mathbb{\dot{Q}}_{\alpha,\beta
}\mid\alpha\leq\omega_{1},\beta<\omega_{2}\mathbb{\rangle}\right)  $$ is a
standard matrix iteration. This follows by the same arguments as in
\cite{UltrafilterswithSmallGeneratingSets} or
\cite{MadFamiliesSplittingFamiliesandLargeContinuum}. We leave the details to
the reader.

We can now prove the following:

\begin{theorem}
There is a model of \emph{ZFC }in which $\mathfrak{a}=\omega_{1}$ and there is
an ultrafilter $\mathcal{W}$ such that $\mathfrak{p}\left(  \mathcal{W}%
\right)  =\mathfrak{c}=\omega_{2}.$
\end{theorem}

\begin{proof}
We start with a model $V$ of the Continuum Hypothesis. Let $G\subseteq
\mathbb{P}_{\omega_{1},\omega_{2}}$ be a generic filter (where $\mathbb{P}%
_{\omega_{1},\omega_{2}}$ is the forcing described above). We will show that
$V\left[  G\right]  $ is the model we are looking for. A straightforward
argument shows that $V\left[  G\right]  \models\mathfrak{c}=\omega_{2}.$

We argue in $V\left[G\right]  .$ Note that $R=\left\{  r_{\beta}\mid
\beta<\omega_{2}\right\}  $ is a decreasing tower, so it is centered.
Let $\mathcal{W}$ be the filter generated by $R.$ It is easy to see that
$\mathcal{W}$ is in fact an ultrafilter (this is because $r_{\beta}$ is a
$\mathbb{M}\left(  \mathcal{U}_{\omega_{1}\beta}\right)  $-generic real, for
more details, the reader may consult
\cite{UltrafilterswithSmallGeneratingSets}). Furthermore, since $\mathcal{W}$
is generated by a tower of length $\omega_{2},$ it follows that $\mathfrak{p}%
\left(  \mathcal{U}\right)  =\omega_{2}.$

It remains to prove that $\mathfrak{a}=\omega_{1}$ holds in $V\left[
G\right]  .$ This is the same argument as the one used in section 4 of
\cite{MadFamiliesSplittingFamiliesandLargeContinuum}. We include the argument
for completeness. We will prove that $\mathcal{A}_{\omega_{1}}=\left\{
A_{\alpha}\mid\alpha<\omega_{2}\right\}  $ is a \textsf{MAD} family in
$V\left[  G\right]  .$ We start with the following:

\begin{claim}
Let $\alpha<\omega_{1}$ and $\beta\leq\omega_{2}.$ Then $\bigstar
_{\mathcal{A}_{\alpha},\text{ }A_{\alpha}}^{V_{\alpha\beta},V_{\left(
\alpha+1\right)  \beta}}$ holds.
\end{claim}

\smallskip

Fix $\alpha<\omega_{1},$ we prove the claim by induction on $\beta.$ The case
$\beta=0$ follows by Lemma \ref{estrellita sucesor}. If the claim is true for
$\beta<\omega_{2},$ then it is also true for $\beta+1$ by point
\ref{paso inductivo para estrellita} in the definition of our iteration.
Finally, let $\beta$ be a limit ordinal and assume that the lemma is true for
every ordinal less than $\beta.$ If $\beta$ has uncountable cofinality, then
there is nothing to prove (note that if $h:\omega\times\left[  \alpha\right]
^{<\omega}\longrightarrow\omega$ is in $V_{\alpha\beta}$, for every
$F\in\left[  \alpha\right]  ^{<\omega}$ we can define $h_{F}:\omega
\longrightarrow\omega$ given by $h_{F}\left(  n\right)  =h\left(  n,F\right)
.$ For every $F\in\left[  \alpha\right]  ^{<\omega},$ there is $\eta<\beta$
such that $h_{F}\in V_{\alpha\eta}$). If $\beta$ has countable cofinality, the
claim follows by the Lemma 12 point 1 of
\cite{MadFamiliesSplittingFamiliesandLargeContinuum}). This proves the claim.

\smallskip

We will now prove the following:

\begin{claim}
$\mathcal{A}_{\omega_{2}}$ is a MAD family in $V\left[  G\right]  .$
\end{claim}

\smallskip 

Let $X\in\mathcal{I}\left(  \mathcal{A}_{\omega_{2}}\right)  ^{+}$ (in
$V\left[  G\right]  $). Since $\mathbb{P}_{\omega_{1},\omega_{2}}$ is a finite
support iteration of the c.c.c. partial orders $\langle\mathbb{P}_{\omega
_{1},\beta}\mid\beta<\omega_{2}\rangle,$ there is $\beta<\omega_{2}$ such that
$X\in V_{\omega_{1}\beta}.$ Furthermore, since we are using a standard matrix
iteration and $X$ is a real, there is $\alpha<\omega_{1}$ such that $X\in
V_{\alpha\beta}.$ Since $\bigstar_{\mathcal{A}_{\alpha},\text{ }A_{\alpha}%
}^{V_{\alpha\beta},V_{\left(  \alpha+1\right)  \beta}}$ holds and
$X\in\mathcal{I}\left(  \mathcal{A}_{\alpha}\right)  ^{+},$ by Lemma
\ref{Estrellita pega}, we have that $A_{\alpha}\cap X$ is infinite. This
finishes the proof.
\end{proof}

	\section{Non-pseudocompact MAD families}
	
	Here we prove that consistently there is a MAD family $\mathcal A$ of size $<\mathfrak c$ such that $\exp(\mathcal A)$ is not pseudocompact. This, of course, also provides a model where pseudocompact MAD families do not exist generically.
	
	Our example will be a MAD family over the infinite countable set $\triangle=\{(m,n)\in \omega\times\omega: n\leq m\}$. The elements of $\mathcal A$ will consist of graphs of partial functions.  The result easily follows from the following:
	
	\begin{theorem}\label{existenceModel} It is consistent with $\mathfrak c>\omega_2$
	that there are MAD families on $\{\mathcal A_\alpha:\alpha<\omega_1\}$ on $\omega$, a
	MAD family $\mathcal A$ of size $\omega_2$ on $\Delta$ consisting of partial functions below the diagonal such that 
	\begin{enumerate}
	    \item $\forall s\in \mathcal A \ \exists \alpha<\omega_1 \  \mathrm{\dom}(s) \in \mathcal A_\alpha$, 
	    \item $s\neq t\in \mathcal A \ \Rightarrow \  \mathrm{\dom}(s)\neq \mathrm{\dom}(t)$.
	\end{enumerate}
	and  for every family $\mathcal F$ of $\omega_1$-many partial functions below the diagonal there is a total function below the diagonal almost disjoint from every element of $\mathcal F$.
	\end{theorem}
	
	We shall postpone the proof of the theorem and first show that it suffices to prove the desired result.
	
	\begin{theorem}
	It is relatively consistent with {\sf ZFC} that there is a MAD family $\mathcal A$ of size $<\mathfrak c$ such that $\exp(\Psi(\mathcal A))$ is not pseudocompact.
	\end{theorem}
	
	\begin{proof} By Theorem \ref{existenceModel}, it is consistent that $\mathfrak c>\omega_2$ and that there exists $\mathcal A$ and $(\mathcal A_\alpha: \alpha<\omega_1)$ as in the statement of the proposition. We show that $\mathcal \exp(\Psi(\mathcal A))$ is not pseudocompact.
	
	Let $F:\omega\rightarrow \exp(\Psi(\mathcal A))$ be given by $F_n=\{(m, n): m\leq n\}$. We claim $F$ has no accumulation point in $\exp(\Psi(\mathcal A))$. Suppose $L$ is such an accumulation point. Then, since $F$ is a sequence of pairwise disjoint finite subsets of $\triangle$, $L\subseteq \mathcal A$.
	
	If $|L|<\omega_2$, there exists a total function below the diagonal $f$ almost disjoint from every element of $L$. Then $L \in (\Psi(\mathcal A)\setminus \cl f)^+$ but $F_n \notin (\Psi(\mathcal A)\setminus \cl f)^+$ for every $n \in \omega$, a contradiction.
	
	Now suppose $|L|=\omega_2$. There exists $\alpha<\omega_1$ such that there exists two distinct $s, t \in \mathcal A$ such that $\dom s, \dom t \in \mathcal A_\alpha$. Since $s, t$ are distinct, it follows that $\dom(s)\neq \dom(t)$, and since $\mathcal A_\alpha$ is an almost disjoint family, $\dom s\cap \dom t\subseteq k$ for some $k \in \omega$. Then $L \in (\{s\}\cup\{s\setminus\{(m, n): m\leq n<k\})^-\cap \{t\}\cup\{t\setminus\{(m, n): m\leq n<k\})^-$, but no element of the sequence $F$ is a member of the latter open set.\end{proof}

Let $\mathcal{A}$ be an \textsf{AD} family. For the reader's convenience we repeat the definition of the \emph{Mathias forcing $\mathbb{M(\mathcal{A})}$}
associated with $\mathcal{A}$ is the set of all
$p=\left(  s_{p},F_{p}\right)$ such that
\begin{enumerate}
\item there is $n_{p}\in\omega$ such that $s_{p}:n_{p}\longrightarrow2$, and 
\item $F_{p}\in\lbrack\mathcal{A}]^{<\omega}.$
\end{enumerate}
ordered by 
$p=\left(  s_{p},F_{p}\right) \leq q=\left(  s_{q},F_{q}\right)$ if 
\begin{enumerate}
\item $s_{q}\subseteq s_{p}$ (hence $n_{q}\leq n_{p}$), $F_{q}\subseteq F_{p}$, and
\item if $B\in F_{q},$ then $B\cap s_{p}^{-1}\left(  1\right)  \subseteq n_{q}.$
\end{enumerate}

Given $p=\left(  s_{p},F_{p}\right)  \in\mathbb{M}\left(  \mathcal{A}\right)
,$ we call $s_{p}$ the \emph{stem of }$p$ and $F_{p}$ the \emph{side condition
of }$p.$ The \emph{length of }$p$ is $len\left(  p\right)  =n_{p}.$ If
$G\subseteq\mathbb{M}\left(  \mathcal{A}\right)  $ is a generic filter, the
\emph{generic real of }$\mathbb{M}\left(  \mathcal{A}\right)  $\emph{ }is
defined as $A_{gen}=\bigcup
\left\{  i\mid\exists\left(  s,F\right)  \in G\left(  s\left(  i\right)
=1\right)  \right\}  .$ The following lemma is well-known and easy to prove:

\begin{lemma}
Let $\mathcal{A}$ be an \textsf{AD} family, $G\subseteq\mathbb{M}\left(
\mathcal{A}\right)  $ a generic filter and $A_{gen}$ the generic real.
\label{Basic Mathias}

\begin{enumerate}
\item $A_{gen}$ is an infinite subset of $\omega.$

\item $A_{gen}$ is almost disjoint with every element of $\mathcal{A}.$

\item For every $X\in\left[  \omega\right]  ^{\omega}\cap V,$ if
$X\in\mathcal{I}\left(  \mathcal{A}\right)  ^{+},$ then $A_{gen}\cap X$ is infinite.
\end{enumerate}
\end{lemma}

By $Fun\left(  \Delta\right)  $ we denote the set of all
functions $f:\omega\longrightarrow\omega$ such that $f\subseteq\Delta.$ Given
$X\in\left[  \omega\right]  ^{\omega},$ define $PFun_{X}\left(  \Delta\right)
$ as the set of all functions $g$ such that there is $A\in\left[  X\right]
^{\omega}$ for which $g:A\longrightarrow\omega$ and $g\subseteq\Delta.$ By
$PFun\left(  \Delta\right)  $ we denote $PFun_{\omega}\left(  \Delta\right)
.$ Note that if $f,g\in PFun\left(  \Delta\right)  $ then $f$ and $g$ are
almost disjoint if and only if the set $\left\{  n\in \dom\left(  f\right)
\cap \dom\left(  g\right)  \mid f\left(  n\right)  =g\left(  n\right)
\right\}  $ is finite.

\begin{definition}
Define $\mathfrak{ie}$ as the smallest size of a family $\mathcal{F}\subseteq
PFun\left(  \Delta\right)  $ such that for every $g\in Fun\left(
\Delta\right)  $ there is $f\in\mathcal{F}$ such that $\left\vert f\cap
g\right\vert =\omega.$
\end{definition}

The cardinal invariant $\mathfrak{ie}$ is closely related (though not equal) to the invariant $\cov^*(\mathcal{ED}_{\mathrm{fin}})$ defined in \cite{PairSplitting}. If $X \in [\omega]^\omega$ and $n \in \omega$, we let $X(n)$ be the $n$-th element of $X$, and $X|_n=\{X(i): i<n\}$ is the set of the first $n$ elements of $X$.

\begin{definition}
Let $X\in\left[  \omega\right]  ^{\omega}$ and $\mathcal{B}\subseteq
PFun\left(  \Delta\right)  .$ Define the forcing $\mathbb{E}_{\Delta}\left(
\mathcal{B},X\right)  $ as the set of all $p=\left(  s_{p},n_{p},F_{p}\right)
$ with the following properties:

\begin{enumerate}
\item $n_{p}\in\omega,$ $F_{p}\in\left[  \mathcal{B}\right]  ^{<\omega}.$

\item $s_{p}:X|_{n_p}\longrightarrow\omega$ and $s_{p}\subseteq\Delta.$

\item $2\left\vert F_{p}\right\vert \leq n_{p}.$
\end{enumerate}

Let $p=\left(  s_{p},n_{p},F_{p}\right)  ,$ $q=\left(  s_{q},n_{q}%
,F_{q}\right)  \in\mathbb{E}_{\Delta}\left(  \mathcal{B}\right)  ,$ we define
$p\leq q$ if the following conditions hold:

\begin{enumerate}
\item $n_{q}\leq n_{p},$ $F_{q}\subseteq F_{p}$ and $s_{q}\subseteq s_{p}.$

\item If $f\in F_{q}$ and $i\in \dom f \cap X_{n_q}\setminus X_{n_p},$ then $s_{p}\left(
i\right)  \neq f\left(  i\right)  $.
\end{enumerate}
\end{definition}

Given $p=\left(  s_{p},n_{p},F_{p}\right)  \in\mathbb{E}_{\Delta}\left(
\mathcal{B},X\right)  ,$ we call $s_{p}$ the \emph{stem of }$p$ and $F_{p}$
the \emph{side condition of }$p.$ Define the \emph{length of }$p$ as
$len\left(  p\right)  =n_{p}.$ By $\mathbb{E}_{\Delta}$ we will denote
$\mathbb{E}_{\Delta}\left(  Fun\left(  \Delta\right)  ,\omega\right)  .$ If
$G\subseteq\mathbb{E}_{\Delta}\left(  \mathcal{B},X\right)  $ is a generic
filter, the \emph{generic real of }$\mathbb{E}_{\Delta}\left(  \mathcal{B}%
,X\right)  $\emph{ }is defined as $f_{gen}=\bigcup
\left\{  s\mid\exists\left(  s,n,F\right)  \in G\right\}  .$ The analogue of
lemma \ref{Basic Mathias} is the following:

\begin{lemma}
Let $X\in\left[  \omega\right]  ^{\omega}$, $\mathcal{B}\subseteq Fun\left(
\Delta\right)  $ and $f_{gen}$ the generic real. \label{Basic E}

\begin{enumerate}
\item $f_{gen}:X\longrightarrow\omega$ and $f_{gen}\subseteq\Delta.$

\item $f_{gen}$ is almost disjoint with every element of $\mathcal{B}.$

\item If $g\in PFun_{X}\left(  \Delta\right)  \cap V$ is such that
$g\in\mathcal{I}\left(  \mathcal{B}\right)  ^{+}$ (where $\mathcal{I}\left(
\mathcal{B}\right)  $ is the ideal generated by $\mathcal{B}$), then
$f_{gen}\cap g$ is infinite.
\end{enumerate}
\end{lemma}

Let $\mathbb{P}$ be a partial order. Recall that a set $L\subseteq\mathbb{P}$
is \emph{linked }if every $p,q\in L$ are compatible. $\mathbb{P}$ is $\sigma
$-linked if $\mathbb{P}$ is the union of countably many linked sets.

\begin{lemma}
Let $X\in\left[  \omega\right]  ^{\omega}$ and $\mathcal{B}\subseteq
Fun\left(  \Delta\right)  .$\label{lemma del 4}

\begin{enumerate}
\item Let $p=\left(  s_{p},n_{p},F_{p}\right)  ,$ $q=\left(  s_{q},n_{q}%
,F_{q}\right)  \in\mathbb{E}_{\Delta}\left(  \mathcal{B},X\right)  .$ If
$s_{p}=s_{q}$ (hence $n_{p}=n_{q}$), then $p$ and $q$ are compatible.

\item $\mathbb{E}_{\Delta}\left(  \mathcal{B},X\right)  $ is $\sigma$-linked.

\item Let $p=\left(  s_{p},n_{p},F_{p}\right)  ,$ $q=\left(  s_{q},n_{q}%
,F_{q}\right)  \in\mathbb{E}_{\Delta}\left(  \mathcal{B},X\right)  .$ If
$s_{p}=s_{q}$ and \linebreak $4\left\vert F_{p}\right\vert ,4\left\vert F_{q}\right\vert
\leq n_{p}$ then $r=\left(  s_{p},n_{p},F_{p}\cup F_{q}\right)  $ extends both $p$
and $q.$
\end{enumerate}
\end{lemma}

\begin{proof}
Let $p=\left(  s_{p},n_{p},F_{p}\right)  ,$ $q=\left(  s_{q},n_{q}%
,F_{q}\right)  \in\mathbb{E}_{\Delta}\left(  \mathcal{B},X\right)  $ with
$s=s_{p}=s_{q}.$ We first find $t\subseteq\Delta$ with the following properties:

\begin{enumerate}
\item $s\subseteq t.$

\item For every $f\in F_{p}\cup F_{q}$ and $i\in \dom\left(  t\right)
\setminus \dom\left(  s\right)  ,$ we have that $t\left(  i\right)  \neq
f\left(  i\right)  .$

\item $\left\vert t\right\vert \geq 2\left\vert F_{p}\cup F_{q}\right\vert .$
\end{enumerate}

We can find such $t$ since $2\left\vert F_{p}\right\vert ,2\left\vert
F_{q}\right\vert \leq n_{p}.$ It follows that $r=\left(  t,F_{p}\cup F_{q}\right)
$ is an extension of both $p$ and $q.$

We can now prove that $\mathbb{E}_{\Delta}\left(  \mathcal{B},X\right)  $ is
$\sigma$-linked. For every $n\in\omega$ and $s:X|_n\longrightarrow\omega$
with $s\subseteq\Delta,$ define $L\left(  s,n\right)  =\left\{  p=\left(
s_{p},n_{p},F_{p}\right)  \mid s_{p}=s\right\}  .$ Clearly each $L\left(
s,n\right)  $ is linked by the previous point and $\mathbb{E}_{\Delta}\left(
\mathcal{B},X\right)  =\bigcup\{
L\left(  s,n\right): n \in \omega, s \subseteq \delta, s \in \omega^{X|_n}\}  .$

Point 3 follows by the definitions.
\end{proof}

The following result was inspired by Lemma 5.1 of A. Miller's 
\cite{SomePropertiesofMeasureandCategory}:

\begin{proposition}
Let $n\in\omega$, $s:n\longrightarrow\omega$ \ with $s\subseteq\Delta.$ Let
$D\subseteq\mathbb{E}_{\Delta}$ be an open dense. There is an antichain $Z\in\left[
D\right]  ^{<\omega}$ such that for every $p=\left(
s,n,F_{p}\right)  \in\mathbb{E}_{\Delta},$ there is $q\in Z$ such that $p$ and
$q$ are compatible.\label{Compacidad E}
\end{proposition}

\begin{proof}
Let $A=\left\{  r_{m}\mid m\in\omega\right\}  \subseteq D$ be a maximal
antichain (note that $A$ is countable since $\mathbb{E}_{\Delta}$ is $\sigma$-linked and therefore  c.c.c.),
let $k=\frac{n}{2}$ in case $n$ is even and $k=\frac{n-1}{2}$ in
case $n$ is odd.

Assume the proposition is false, so for every $m\in\omega,$ there is
$p_{m}=\left(  s,n,F_{m}\right)  \in\mathbb{E}_{\Delta}$ such that $p_{m}\perp
r_{i}$ for each $i\leq m.$ We can assume that each $F_{m}$ has size $k$, let
$F_{m}=\left\{  f_{i}^{m}\right\}  _{i<k}.$ We may view $B=\left\{  F_{m}\mid
m\in\omega\right\}  $ as a subset of $Fun\left(  \Delta\right)  ^{k}.$ Since
$Fun\left(  \Delta\right)  ^{k}$ is a compact space, we can find $F=\left\{
g_{i}\right\}  _{i<k}$ an accumulation point of $B.$

Let $p=\left(  s,n,F\right)  ,$ since $A$ is a maximal antichain, there is
$j\in\omega$ such that $p$ and $r_{j}$ are compatible. Let $q=\left(
t,l,G\right)  $ be a common extension of both of them. Since $F$ is an
accumulation point of $B,$ there is $m>l,j$ such that $f_{i}^{m}%
\upharpoonright l=g_{i}\upharpoonright l$ for every $i<k.$ Let $\overline
{p}_{m}=\left(  t,l,F_{m}\right)  $ and note that $\overline{p}_{m}\leq
p_{m}.$ It follows that $\overline{p}_{m}$ and $q$ are compatible, in
particular, $p_{m}$ and $q$ are compatible, which implies that $p_{m}$ and
$r_{j}$ are compatible, which is a contradiction.
\end{proof}

For the rest of the section, we fix sets $\left\{  D_{\gamma}\mid\gamma
\in\omega_{1}\right\}  ,$ $H,$ $E$ and a function $R$ with the following properties:

\begin{enumerate}
\item $\left\{  H,E\right\}  \cup\left\{  D_{\gamma}\mid\gamma\in\omega
_{1}\right\}  $ is a partition of $\omega_{2}.$

\item For every $\gamma\in\omega_{1},$ we have that $\left\vert D_{\gamma
}\right\vert =\left\vert H\right\vert =\left\vert E\right\vert =\omega_{2}.$

\item $R:%
{\textstyle\bigcup\limits_{\gamma\in\omega_{1}}}
D_{\gamma}\longrightarrow H$ is a bijective function such that $\alpha<R\left(  \alpha\right)  $ for
every $\alpha\in%
{\textstyle\bigcup\limits_{\gamma\in\omega_{1}}}
D_{\gamma}.$
\end{enumerate}

We now define a finite support iteration $\langle\mathbb{P}_{\alpha
},\mathbb{\dot{Q}}_{\alpha}\mid\alpha\leq\omega_{2}\rangle$ as follows:

\begin{enumerate}
\item If $\alpha\in E,$ then $\mathbb{P}_{\alpha}\Vdash``\mathbb{\dot{Q}%
}_{\alpha}=\mathbb{E}_{\Delta}\mrq.$

\item For every $\gamma\in\omega_{1}$ and $\xi\in D_{\gamma},$ let $\dot
{A}_{\gamma}^{\xi}$ be a name for the $(\mathbb{M(}\mathcal{A}_{\gamma}^{\xi
}),V_{\xi})$-generic real (where $\mathcal{A}_{\gamma}^{\xi}=\{\dot{A}%
_{\gamma}^{\eta}\mid\eta\in\xi\cap D_{\gamma}\}$ and $V_{\xi}$ is the
extension by $\mathbb{P}_{\xi}$).

\item If $\alpha\in D_{\gamma}$ (with $\gamma\in\omega_{1}$), then
$\mathbb{P}_{\alpha}\Vdash``\mathbb{\dot{Q}}_{\alpha}=\mathbb{M(}%
\mathcal{A}_{\gamma}^{\alpha})\mrq.$

\item Given $\xi\in H,$ let $\gamma\in\omega_{1}$ and $\beta\in D_{\gamma}$
such that $\xi=R\left(  \beta\right)  .$ let $\dot{f}_{\xi}$ be a name for the
$(\mathbb{E}_{\Delta}\left(  \mathcal{B}_{\xi},A_{\gamma}^{\beta}\right)
,V_{\xi})$-generic real (where $\mathcal{B}_{\xi}=\{\dot{f}_{\eta}\mid\eta
\in\xi\cap H\}$ ).

\item If $\alpha\in H,$ with ($R\left(  \beta\right)  =\alpha$ and $\beta\in
D_{\gamma}$) then $\mathbb{P}_{\alpha}\Vdash``\mathbb{\dot{Q}}_{\alpha
}=\mathbb{E}_{\Delta}\left(  \mathcal{B}_{\alpha},A_{\gamma}^{\beta}\right)
\mrq.$
\end{enumerate}

If $p\in\mathbb{P}_{\alpha}$ and $\dot{x}$ is a $\mathbb{P}_{\alpha}$-name for
a condition of $\mathbb{\dot{Q}}_{\alpha},$ we denote by $p^{\frown}\dot{x}$
the condition $r\in\mathbb{P}_{\alpha+1}$ such that $r\upharpoonright\alpha=p$
and $r\left(  \alpha\right)  =\dot{x}.$

We will need to develop some combinatorial tools for our forcing in order to
prove the main result. Let $\alpha\leq\omega_{2},$ we say that a condition
$p\in\mathbb{P}_{\alpha}$ is \emph{pure }if there is $n\in\omega$ such that
for every $\xi\in \dom\left(  p\right)  ,$ the following holds:

\begin{enumerate}
\item If $\xi\in D_{\gamma}$ (for some $\gamma\in\omega_{1}$), then there is
$s_{\xi}\in2^{n}$ and $J_{\xi}\in\left[  D_{\gamma}\cap\xi\right]  ^{<\omega}$
such that $p\left(  \xi\right)  =(s_{\xi},\{\dot
{A}_{\gamma}^{\eta}\mid\eta\in J_{\xi}\})$.

\item Furthermore, $J_{\xi}\subseteq \dom\left(  p\right)  .$

\item If $\xi\in H$ and $\beta$ is such that $R\left(  \beta\right)  =\xi,$
then $\beta\in \dom\left(  p\right)  .$

\item If $\xi\in H,$ (let $\beta$ such that $R\left(  \beta\right)  =\xi$),
then there is $z_{\xi}:s_{\beta}^{-1}\left(  1\right)  \longrightarrow\omega$
with $z_{\xi}\subseteq\Delta$ and $J_{\xi}\in\left[  H\cap\xi\right]
^{<\omega}$ such that $p\left(  \xi\right)
=(z_{\xi},n,\{\dot{f}_{\eta}\mid\eta\in J_{\xi}\})$ and
$4\left\vert J_{\xi}\right\vert \leq n$ (where $s_{\beta}$ is defined as in point 1).

\item Furthermore, $J_{\xi}\subseteq \dom\left(  p\right)  .$

\item If $\xi\in E,$ then there is $m_{\xi}\in\omega\ $, $z_{\xi}:m_{\xi
}\longrightarrow\omega$ with $z_{\xi}\subseteq\Delta$ and $\dot{J}$ such that
$p\left(  \xi\right)  =(z_{\xi},m_{\xi},\dot
{J})$ and there is $k_\xi$ such that $4k_\xi\leq m_\xi$ and $\mathbb P_\xi$-names $\rho_0, \dots, \rho_k$ such that $\dot J=\{(\rho_0, \mathbbm 1_{\mathbb P_\xi}), \dots, (\rho_{k_\xi-1}, \mathbbm 1_{\mathbb P_\xi})\}$
\end{enumerate}

Of course, by $(., ., .)$ we are denoting a name for the triple. Given a pure condition $p$ and $\beta \in \dom p$, we denote by $len(p)$ the size of the first coordinate of $p$.

In the above definition, recall that $\dot{A}_{\gamma}^{\xi}$ is the name for
the $(\mathbb{M(}\mathcal{A}_{\gamma}^{\xi}),V_{\xi})$-generic real and
$\dot{f}_{\xi}$ is the name for the $(\mathbb{E}_{\Delta}\left(
\mathcal{B}_{\xi}\right)  ,V_{\xi})$-generic real. An important difference
between point 4 and point 6 is that we may have $m_{\xi}\neq n.$ We call $n$
\emph{the height of }$p.$ One of the purposes of pure conditions is to avoid
(as much as possible) the use of names and use real objects. We now have the following:

\begin{lemma}
Pure conditions are dense in $\mathbb{P}_{\alpha}.$
\end{lemma}

\begin{proof}
We prove the lemma by induction on $\alpha.$ The cases where $\alpha=0$ or
$\alpha$ is limit are straightforward, so we focus on the successor case.
Assume the lemma is true for $\alpha,$ we will prove it is also true for
$\alpha+1.$ Let $p\in\mathbb{P}_{\alpha+1},$ we may assume that $\alpha\in
\dom\left(  p\right)  .$

\begin{case}
$\alpha\in E.$
\end{case}

First, we find $p_{1}\leq p\upharpoonright\alpha$ such that there are
$m_{\alpha}\in\omega\ $, $z_{\alpha}:m_{\alpha}\longrightarrow\omega$ with
$z_{\alpha}\subseteq\Delta$ and $\dot{L}$ such that $p_{1}\Vdash``p\left(
\alpha\right)  =(z_{\alpha},m_{\alpha},\dot{L})\mrq.$ By extending $p$ and $p_1$, we may
even assume that $p_{1}\Vdash``4\left\vert \dot{L}\right\vert \leq m_{\alpha
}\mrq$. So we may find $p_2\leq p_1$, $k_\xi\leq \frac{m_\xi}{4}$ and names $\rho_0, \dots, \rho_{k_\xi-1}$ such that $p_2\Vdash \dot L=\{\rho_0, \dots, \rho_{k_\xi-1}\}$. Let $\dot J=\{(\rho_0, \mathbbm 1_{\mathbb P_\xi}), \dots, (\rho_{k_\xi-1}, \mathbbm 1_{\mathbb P_\xi})\}$. By the inductive hypothesis, let $q\leq p_{2}$ be a pure
condition. Define $\overline{q}\in\mathbb{P}_{\alpha+1}$ such that the
following holds:

\begin{enumerate}
\item $\overline{q}\upharpoonright\alpha=q.$

\item $\overline{q}\left(
\alpha\right)  =(z_{\alpha},m_{\alpha},\dot{J})$.
\end{enumerate}

It is easy to see that $\overline{q}$ is a pure extension of $p.$

\begin{case}
$\alpha\in D_{\gamma}$ (for some $\gamma\in\omega_{1}$).
\end{case}

First, we find $p_{1}\leq p\upharpoonright\alpha$ such that there are
$m\in\omega,$ $s\in2^{m}$ and $J_{\alpha}\in\left[  D_{\alpha}\cap
\alpha\right]  ^{<\omega}$ such that $p_{1}\Vdash``p\left(  \alpha\right)
=(s,\{\dot{A}_{\gamma}^{\eta}\mid\eta\in J_{\alpha}\})\mrq,$ we
may assume that $J_{\alpha}\subseteq \dom\left(  p_{1}\right)  .$ By the
inductive hypothesis, let $q\leq p_{1}$ be a pure condition, let $n$
witnessing that $q$ is pure, without lost of generality, we may assume that
$m<n.$ Let $s_{\alpha}\in2^{n}$ such that $s_{\alpha}\upharpoonright m=s$ and
$s_{\alpha}\left(  i\right)  =0$ for every $i\in\lbrack m,n).$ Define
$\overline{q}\in\mathbb{P}_{\alpha+1}$ such that the following holds:

\begin{enumerate}
\item $\overline{q}\upharpoonright\alpha=q.$

\item $\overline{q}\left(
\alpha\right)  =(s_{\alpha},\{\dot{A}_{\gamma}^{\eta}\mid\eta\in J_{\alpha
}\})$.
\end{enumerate}

It is easy to see that $\overline{q}$ is a pure extension of $p.$

\begin{case}
$\alpha\in H.$
\end{case}

First, we find $p_{1}\leq p\upharpoonright\alpha$ such that there are
$m\in\omega,$ $z_{\alpha}:m\longrightarrow\omega$ with $z_{\alpha}%
\subseteq\Delta$ and $J_{\alpha}\in\left[  H\cap\alpha\right]  ^{<\omega}$
such that $p_{1}\Vdash``p\left(  \alpha\right)  =(z_{\alpha},m,\{\dot{f}%
_{\eta}\mid\eta\in J_{\alpha}\})\mrq,$ we may also assume that
$4\left\vert J_{\alpha}\right\vert <m$ and that $J_{\alpha}\subseteq
\dom\left(  p_{1}\right)  .$ By the inductive hypothesis, let $q\leq p_{1}$ be
a pure condition, let $n$ witnessing that $q$ is pure, without lost of
generality, we may assume that $m<n$ and $J_{\alpha}\subseteq \dom\left(
q\right)  .$ Let $z_{\alpha}:n\longrightarrow\omega$ such that $z_{\alpha
}\subseteq\Delta,$ $z_{\alpha}\upharpoonright m=s$ and $z_{\alpha}\left(
i\right)  \neq z_{\xi}\left(  i\right)  $ for every $i\in\lbrack m,n)$ and
$\xi\in J_{\alpha}.$ Define $\overline{q}\in\mathbb{P}_{\alpha+1}$ such that
the following holds:

\begin{enumerate}
\item $\overline{q}\upharpoonright\alpha=q.$

\item $\overline{q}\left(
\alpha\right)  =(z_{\alpha},n,\{\dot{f}_{\eta}\mid\eta\in J_{\alpha
}\})$.
\end{enumerate}

It is easy to see that $\overline{q}$ is a pure extension of $p.$
\end{proof}

\begin{lemma}
Let $\alpha\leq\omega_{2},$ $p\in\mathbb{P}_{\alpha}$ a pure condition and
$m\in\omega.$ There is $q\in\mathbb{P}_{\alpha}$ with the following
properties: \label{pure crecer stems}

\begin{enumerate}
\item $q\leq p.$

\item $q$ is pure.

\item If $\beta\in \dom\left(  q\right)  $ then $m\leq len\left(  q\left(
\beta\right)  \right)  .$
\end{enumerate}
\end{lemma}

\begin{proof}
We prove the lemma by induction on $\alpha.$ The cases where $\alpha=0$ or
$\alpha$ is limit are straightforward, so we focus on the successor case.
Assume the lemma is true for $\alpha,$ we will prove it is also true for
$\alpha+1.$ Let $p\in\mathbb{P}_{\alpha+1},$ we may assume that $\alpha\in
\dom\left(  p\right)  .$

\begin{case}
$\alpha\in E.$
\end{case}

Suppose $p\left(  \alpha\right)  =(z_{\alpha},m_{\alpha},\dot{J}).$ In case that
$m\leq m_{\alpha},$ we apply the inductive hypothesis to $p\upharpoonright
\alpha$ and we are done. Assume that $m_{\alpha}<m.$ By the inductive
hypothesis, we may find $q\leq p\upharpoonright\alpha$ such that the following holds:

\begin{enumerate}
\item $q$ is pure.

\item If $\beta\in \dom\left(  q\right)  $ then $q(\beta)\geq m$.

\item For every $j<k_\xi$ there is $w_{j}:m\longrightarrow
\omega$ such that $q\Vdash``\rho_j\upharpoonright m=w_{j}%
\mrq.$
\end{enumerate}

We now define $s:m\longrightarrow\omega,$ with $s\subseteq\Delta$ such that
$z_{\alpha}\subseteq s$ and $s\left(  i\right)  \neq w_{j}\left(
i\right)  $ for every $i\in(m_{\alpha},m]$ and $j<n$. It is
clear that $q^{\frown}(s,m,\dot{J})$ has the desired properties.

\begin{case}
$\alpha\in D_{\gamma}$ (for some $\gamma\in\omega_{1}$).
\end{case}

Suppose $p\left(  \alpha\right)  =(s_{\alpha},\{\dot A_\gamma^\eta: \eta \in \dot{J}_{\alpha}\})$ and $n$ is such that
$s_{\alpha}:n\longrightarrow2.$ By the inductive hypothesis, we may find
$q\leq p\upharpoonright\alpha$ such that the following holds:

\begin{enumerate}
\item $q$ is pure.

\item If $\beta\in \dom\left(  p\right)  $ then $q(\beta)\geq \max\{m, n\}$.
\end{enumerate}

Let $k$ be the height of $q.$ We now define $z:k\longrightarrow2$ such that
$s_{\alpha}\subseteq z$ and $z\left(  i\right)  =0$ for every $i\in\lbrack
n,k).$ It is clear that $q^{\frown}(z,\dot{J}_{\alpha})$ has the desired properties.

\begin{case}
$\alpha\in H.$
\end{case}

Similar to the previous cases.
\end{proof}

\begin{definition}
Let $\alpha\leq\omega_{2}$ and $p\in\mathbb{P}_{\alpha}$ a pure condition. We
say that $p$ has the \emph{descending condition }if for every $\beta_{1}%
,\beta_{2}\in \dom\left(  p\right)  \cap E,$ if $\beta_{1}<\beta_{2},$ then
$len\left(  p\left(  \beta_{1}\right)  \right)  \geq len\left(  p\left(
\beta_{2}\right)  \right)  .$
\end{definition}

Using the previous lemma and induction, we get the following:

\begin{lemma}
For every $\alpha\leq\omega_{2},$ pure conditions with the descending
condition are dense. \label{pure descending}
\end{lemma}

\begin{proof}
We prove the lemma by induction on $\alpha.$ The cases where $\alpha=0$ or
$\alpha$ is limit are straightforward, so we focus on the successor case.
Assume the lemma is true for $\alpha,$ we will prove it is also true for
$\alpha+1.$ Let $p\in\mathbb{P}_{\alpha+1}$ be a pure condition, we may assume
that $\alpha\in \dom\left(  p\right)  .$ In case $\alpha\notin E,$ there is
nothing to do, so assume that $\alpha\in E.$

Let $p\left(  \alpha\right)  =\left(  s,n,\dot{J}\right)  ,$ by the inductive
hypothesis and Lemma \ref{pure crecer stems}, we can find $q\in\mathbb{P}%
_{\alpha}$ such that $q\leq p\upharpoonright\alpha,$ $q$ is pure with the descending condition and all the
stems in $q$ have size larger than $n$. It is clear that $q^{\frown}\left(  s,n,\dot{J}\right)  $ is the condition we are looking for.
\end{proof}

Although pure conditions are the nicest to work with, we will need to deal
with non-pure conditions for some arguments. We will develop the tools needed
in order to do this. First, we will recall a well known forcing lemma that
will be often used implicitly (for a proof, see Lemma 1.19 in the first
chapter of \cite{ProperandImproper}):

\begin{lemma}
Let $\mathbb{P}$ be a partial order, $A=\left\{  p_{\alpha}\mid\alpha\in
\kappa\right\}  \subseteq\mathbb{P}$ a maximal antichain and $\{\dot
{x}_{\alpha}\mid\alpha\in\kappa\}$ be a set of $\mathbb{P}$-names. There is a
$\mathbb{P}$-name $\dot{y}$ such that $p_{\alpha}\Vdash``\dot{y}=\dot
{x}_{\alpha}\mrq$ for every $\alpha\in\kappa.$
\end{lemma}

Let $A\in\left[  E\right]  ^{<\omega}$ and $K:A\longrightarrow\omega^{<\omega
}.$ We will say that $K$ is \emph{suitable }if $K\left(  \alpha\right)
\subseteq\Delta$ for every $\alpha\in A.$ Let $K:A\longrightarrow
\omega^{<\omega}$ be suitable. We say that $q\in\mathbb{P}_{\omega_{2}}$
\emph{follows }$K$ if the following holds:

\begin{enumerate}
\item $A\subseteq \dom\left(  q\right)  .$

\item If $\alpha\in A$ then $q\upharpoonright\alpha\Vdash``q\left(
\alpha\right)  =\left(  K\left(  \alpha\right)  ,\left\vert K\left(
\alpha\right)  \right\vert ,\dot{F}\right)  \mrq$ (for some
$\dot{F}$).
\end{enumerate}

\begin{definition}
Let $A\in\left[  E\right]  ^{<\omega}$ and $K:A\longrightarrow\omega^{<\omega
}$ be suitable. We say that $p\in\mathbb{P}_{\alpha}$ has the $K$%
\emph{-descending condition }if the following holds:

\begin{enumerate}
\item For every $\beta_{1},\beta_{2}\in\left(  \dom\left(  p\right)  \setminus
A\right)  \cap E$, if $\beta_{1}<\beta_{2},$
then $p\upharpoonright\beta
_{2}\Vdash``len\left(  p\left(  \beta_{1}\right)  \right)  \geq len\left(
p\left(  \beta_{2}\right)  \right)  \mrq$.

\item For every $\beta_{1},\beta_{2}\in \dom\left(  p\right)  \cap H,$ if
$\beta_{1}<\beta_{2},$ then $p\upharpoonright\beta_{2}\Vdash``len\left(
p\left(  \beta_{1}\right)  \right)  \geq len\left(  p\left(  \beta_{2}\right)
\right)  \mrq$.

\item For every $\gamma\in\omega_{1}$ and for every $\beta_{1},\beta_{2}\in
\dom\left(  p\right)  \cap D_{\gamma},$ if $\beta_{1}<\beta_{2},$ then
$p\upharpoonright\beta_{2}\Vdash``len\left(  p\left(  \beta_{1}\right)
\right)  \geq len\left(  p\left(  \beta_{2}\right)  \right)
\mrq$.

\item If $\beta=min\left(  \dom\left(  p\right)  \right)  ,$ then there is
$s\in\omega^{<\omega}$ such that $s$ is the stem of $p\left(  \beta\right)  $
(i.e., the stem of $p\left(  \beta\right)  $ is a real object, not just a
name) and for every $\eta\in \dom\left(  p\right)  \setminus A,$ we have that
$p\upharpoonright\eta\Vdash``len\left(  p\left(  \beta\right)  \right)  \geq
len\left(  p\left(  \eta\right)  \right)  \mrq.$
\end{enumerate}
\end{definition}

This new notion does not clash with our previous terminology, since pure
conditions with the descending condition (essentially) satisfy the $\emptyset
$-descending condition. We now introduce the following notions:

\begin{definition}
Let $\alpha\in\omega_{2},$\ $A\in\left[  E\cap\alpha\right]  ^{<\omega}$ and
$K:A\longrightarrow\omega^{<\omega}$ suitable.

\begin{enumerate}
\item Let $\mathbb{P}_{\alpha}^{K}$ be the set of all $p\in\mathbb{P}_{\alpha
}$ such that the following conditions hold:

\begin{enumerate}
\item $p$ follows $K.$

\item $p$ satisfies the $K\,$-descending condition.

\item For every $\beta\in \dom\left(  p\right)  \cap\left(  H\cup E\right)  ,$
if $p\left(  \beta\right)  =\left(  \dot{s},\dot{m},\dot{F}\right)  ,$ then
$p\upharpoonright\beta\Vdash``4\left\vert \dot{F}\right\vert \leq\dot
{m}\mrq.$
\end{enumerate}
\end{enumerate}
\end{definition}

The following result is similar to Lemma \ref{pure crecer stems}:

\begin{lemma}
Let $\alpha\leq\omega_{2}$, $A\in\left[  E\cap\alpha\right]  ^{<\omega
},\,\ K:A\longrightarrow\omega^{<\omega}$ suitable$,$ $p\in\mathbb{P}_{\alpha
}^{K}$ and $m\in\omega.$ There is $q$ such that the following holds:
\label{crecer stems}

\begin{enumerate}
\item $q\in\mathbb{P}_{\alpha}^{K}.$

\item $\dom\left(  q\right)  =\dom\left(  p\right)  .$

\item $q\leq p.$

\item If $\beta\in A,$ then $q\left(  \beta\right)  =p\left(  \beta\right)  .$

\item If $\beta\in \dom\left(  q\right)  \setminus A$ then $q\upharpoonright
\beta \Vdash len\left(  q\left(  \beta\right)  \right)  =max\left\{
m,len\left(  p\left(  \beta\right)  \right)  \right\}  .$
\end{enumerate}
\end{lemma}

\begin{proof}
Note that the last point already implies that $q$ satisfies the $K$-descending
condition. We proceed by induction, the cases $\alpha=0$ and $\alpha$ is limit
are immediate. Assume the lemma is true for $\alpha,$ we will now prove it for
$\alpha+1.$ We may assume that $\alpha\in \dom\left(  p\right)  .$

\begin{case}
$\alpha\notin H\cup E.$
\end{case}

Note that in particular, $\alpha\notin A.$ Let $p\upharpoonright\alpha \Vdash p\left(  \alpha\right)
=\left(  \dot{s},\dot{F}\right)  ,$ by the inductive hypothesis, there is
$q\leq p\upharpoonright\alpha$ as in the lemma. Let $\dot{k}$ be a
$\mathbb{P}_{\alpha}$-name for a natural number, such that $q\Vdash``\dot
{s}:\dot{k}\longrightarrow2\mrq.$ Let $\dot{z}$ be a
$\mathbb{P}_{\alpha}$-name such that $q$ forces the following:

\begin{enumerate}
\item $\dom\left(  \dot{z}\right)  =max\{m,\dot{k}\}.$

\item $\dot{s}\subseteq\dot{z}.$

\item If $i\in \dom\left(  \dot{z}\right)  \setminus \dom\left(  \dot{s}\right)
,$ then $\dot{z}\left(  i\right)  =0.$
\end{enumerate}

It is clear that $q^{\frown}(\dot{z},\dot{F})$ is the condition we were
looking for.

\begin{case}
$\alpha\in H.$
\end{case}

Let $\alpha\in H,$ $\gamma\in\omega_{1}$ and $\beta\in D_{\gamma}$ such that
$R\left(  \beta\right)  =\alpha.$ Let $p\in\mathbb{P}_{\alpha+1}^{K}$ with
$\alpha\in \dom\left(  p\right)  .$ By the inductive hypothesis, we may assume
that \thinspace$p\upharpoonright\alpha$ satisfy the properties in the
conclusion of the lemma$.$ Let $p\upharpoonright\alpha\Vdash p\left(  \alpha\right)  =(\dot{s},\dot{k}%
,\dot{F})$ and find $\dot{n}$ a $\mathbb{P}_{\alpha}$-name for $max\{\dot
{k},m\}.$ Let $\dot{z}$ be a $\mathbb{P}_{\alpha}$-name for a partial function
forced to have the following properties:

\begin{enumerate}
\item $\dot{z}\subseteq\Delta.$

\item $\dot{s}\subseteq\dot{z}.$

\item $\dom\left(  \dot{z}\right)  =\dot{A}_{\gamma}^{\beta}\cap\dot{n}$

\item for all $i\in \dom\left(  \dot{z}\right)  ,$ if $i\notin \dom\left(
\dot{s}\right)  ,$ then $\dot{z}\left(  i\right)  =min\left\{  j\mid\forall
g\in\dot{F}\left(  g\left(  i\right)  \neq j\right)  \right\}  .$
\end{enumerate}

It is clear that $p\upharpoonright\alpha^{\frown}\left(  \dot{z},\dot{n}%
,\dot{F}\right)  $ has the desired properties.

\begin{case}
$\alpha\in E$ and $\alpha\notin A.$
\end{case}

Similar to the previous case.

\begin{case}
$\alpha\in E$ and $\alpha\in A.$
\end{case}

Let $A_{1}=A\setminus\left\{  \alpha\right\}  $ and $K_{1}=K\upharpoonright
A_{1}$. By the inductive hypothesis (applied to $p\upharpoonright\alpha$ and
$K_{1}$) let $q\leq p\upharpoonright\alpha$ as in the lemma. It is easy to see
that $q^{\frown}p\left(  \alpha\right)  $ has the desired properties.
\end{proof}

We will need the following result, which is the generalization of
\ref{Compacidad E} for the iteration:

\begin{lemma}
Let $\alpha\leq\omega_{2},$ $D\subseteq\mathbb{P}_{\alpha}$ an open dense set,
$A\in\left[  E\cap\alpha\right]  ^{<\omega}$ and $K:A\longrightarrow
\omega^{<\omega}$ suitable. If $p\in\mathbb{P}_{\alpha}^{K}$, then there is
$q$ with the following properties: \label{Compacidad iteracion}

\begin{enumerate}
\item $q\in\mathbb{P}_{\alpha}^{K}$

\item $q\leq p.$

\item If $\beta\in A,$ then $q\left(  \beta\right)  =p\left(  \beta\right)  $.

\item There is $L\in\left[  D\right]  ^{<\omega}$ an antichain such that for
every $r\leq q,$ if $r$ follows $K,$ then $r$ is compatible with an element of
$L.$
\end{enumerate}
\end{lemma}

\begin{proof}
We prove the lemma by induction on $\alpha.$ The case where $\alpha=0$ is
clear. We will now prove it for $\alpha+1.$

\begin{case}
$\alpha\notin A.$
\end{case}

Define $\overline{D}$ as the set of all $q\in\mathbb{P}_{\alpha}$ for which
there exists $\overline{q}\in\mathbb{P}_{\alpha+1}$ with the following properties:

\begin{enumerate}
\item $\overline{q}\upharpoonright\alpha=q.$

\item $\overline{q}\in D.$

\item $q\Vdash``\overline{q}\left(  \alpha\right)  \leq p\left(
\alpha\right)  \mrq.$

\item There is $m_{q}\in\omega$ such that $q\Vdash``len\left(  \overline
{q}\left(  \alpha\right)  \right)  =m_{q}\mrq.$

\item In case $\alpha\in H\cup E,$ if $ q\Vdash\overline{q}\left(  \alpha\right)
=\left(  \dot{s},m_{q},\dot{F}\right)  ,$ then $ q\Vdash``4\left\vert \dot
{F}\right\vert \leq m_{q}\mrq.$
\end{enumerate}

It is easy to see that $\overline{D}$ is an open dense subset of
$\mathbb{P}_{\alpha}.$ By the inductive hypothesis, there is $\overline{p}\leq
p\upharpoonright\alpha$ as in the lemma, let $L\in\left[  \overline{D}\right]
^{<\omega}$ an antichain such that for every $q\leq\overline{p},$ if $q$
follows $K,$ then $q$ is compatible with an element of $L.$ Let $L=\left\{
q_{i}\mid i<k\right\}  $ for some $k\in\omega.$ For every $i<k,$ fix
$\overline{q}_{i}\in D$ as in the definition of $\overline{D}.$ Let $\beta
_{0}=min\left(  \dom\left(  p\right)  \right)  ,$ we now find $m\in\omega$ such
that $m>len\left(  p\left(  \beta_{0}\right)  \right)  \ $as well
as$\ m>m_{q_{i}}$ for every $i<k.$ Since $L$ is an antichain, we can find
$\dot{x}$ a $\mathbb{P}_{\alpha}$-name for an element of $\mathbb{\dot{Q}%
}_{\alpha}$ with the following properties:

\begin{enumerate}
\item $q_{i}\Vdash``\dot{x}=\overline{q}_{i}\left(  \alpha\right)
\mrq$ for every $i<k.$

\item $r\Vdash``\dot{x}=p\left(  \alpha\right)  \mrq$ for every
$r$ incompatible with every $q_{i}.$
\end{enumerate}

We now apply lemma \ref{crecer stems} to find $p_{1}$ with the following properties:

\begin{enumerate}
\item $p_{1}\in\mathbb{P}_{\alpha}^{K}.$

\item $p_{1}\leq\overline{p}.$

\item $\dom\left(  p_{1}\right)  =\dom\left(  \overline{p}\right)  .$
\item If $\gamma \in A$, then $p_1(\gamma)=\bar p(\gamma)$.

\item If $\beta\in \dom\left(  p_{1}\right)  \setminus A,$ then $p_{1}%
\upharpoonright\beta\Vdash``len\left(  p_{1}\left(  \beta\right)  \right)
=max\left\{  m,len\left(  \overline{p}\left(  \beta\right)  \right)  \right\}
\mrq.$
\end{enumerate}

Let $q=p_{1}$ $^{\frown}\dot{x}.$ We claim that $q$ has the desired
properties. In order to prove that $q\in\mathbb{P}_{\alpha+1}^{K},$ we only
need to prove that $q$ has the $K$-descending condition (the other properties
are true by definition). Note that $p_{1}$ forces that the length of the stem
of $\dot{x}$ is at most $m$ (since $p\in\mathbb{P}_{\alpha+1}^{K},$ then
$len\left(  p\left(  \alpha\right)  \right)  $ is forced to be at most
$len\left(  p\left(  \beta_{0}\right)  \right)  ,$ which is smaller than $m).$
Since the length of the stem in all the elements of $\dom\left(  p_{1}\right)
\setminus A$ is at least $m,$ it follows that $q$ has the $K$-descending
condition. Clearly $q\leq p$ and if $\beta\in A,$ then $q\left(  \beta\right)
=p\left(  \beta\right)  .$

Finally, let $L_{1}=\left\{  \overline{q}_{i}\mid i<k\right\}  \subseteq D$
and let $r\leq q$ be a condition following $K.$ We need to prove that $r$ is
compatible with an element of $L_{1}.$ Since $r\upharpoonright\alpha\leq
q\upharpoonright\alpha$ and it follows $K,$ we know there is $q_{i}\in L$ such
that $r\upharpoonright\alpha$ and $q_{i}$ are compatible. We claim that $r$
and $\overline{q}_{i}$ are compatible.

Let $r_{1}\in\mathbb{P}_{\alpha}$ be a common extension of both
$r\upharpoonright\alpha$ and $q_{i}.$ Define $\overline{r}=r_{1}^{\text{
\ }\frown}r\left(  \alpha\right)  ,$ we will prove that $\overline{r}$ extends
both $r$ and $\overline{q}_{i}.$ Clearly $\overline{r}\leq r$ and in order to
show that $\overline{r}\leq\overline{q}_{i},$ we only need to prove that
$r_{1}\Vdash``r\left(  \alpha\right)  \leq\overline{q}_{i}\left(
\alpha\right)  \mrq.$ Since $r_{1}\leq q_{i},$ we have that
$r_{1}\Vdash``\dot{x}=\overline{q}_{i}\left(  \alpha\right)
\mrq.$ We also know that $r\upharpoonright\alpha\Vdash``r\left(
\alpha\right)  \leq\dot{x}\mrq,$ we conclude that $r_{1}%
\Vdash``r\left(  \alpha\right)  \leq\overline{q}_{i}\left(  \alpha\right)
\mrq$ and we are done.

\begin{case}
$\alpha\in A$ (in particular, $\alpha\in E$).
\end{case}

Let $s=K\left(  \alpha\right)  $ and $n=\left\vert s\right\vert .$ In this
way, we have that $p\upharpoonright\alpha\Vdash``p\left(  \alpha\right)
=(s,n,\dot{F})\mrq$ for some $\dot{F}.$ Let $G\subseteq
\mathbb{P}_{\alpha}$ be a generic filter with $p\upharpoonright\alpha\in G.$
In $V\left[  G\right]  ,$ we define the set $\overline{D}=\{\dot{x}\left[
G\right]  \mid\exists q\leq p\upharpoonright\alpha\left(  q\in G\wedge
q^{\frown}\dot{x}\in D\right)  \}.$ It is easy to see that $\overline{D}$ is
an open dense subset of $\mathbb{E}_{\Delta}.$ By proposition
\ref{Compacidad E}, there is $Z\in\left[  \overline{D}\right]  ^{<\omega}$ an
antichain such that for every $x=\left(  s,n,J\right)  \in\mathbb{E}_{\Delta
},$ there is $z\in Z$ such that $x$ and $z$ are compatible.

Back in $V,$ define $B$ as the set of all $r\in\mathbb{P}_{\alpha}$ with the
following properties:

\begin{enumerate}
\item Either $r$ and $p\upharpoonright\alpha$ are incompatible, or

\item There are $k\in\omega$ and $Y^{r}=\{\dot{x}_{i}^{r}\mid i<k\}$ such that
$r\Vdash``\dot{Z}=\{\dot{x}_{i}^{r}[\dot{G}]\mid i<k\}\mrq$

\item $r^{\frown}\dot{x}_{i}^{r}\in D$ for every $i<k.$
\end{enumerate}

It is easy to see that $B$ is an open dense subset of $\mathbb{P}_{\alpha}.$
Let $K_{1}=K\upharpoonright\alpha.$ We apply the inductive hypothesis with
$p\upharpoonright\alpha,$ $B$ and $K_{1}.$ In this way, there are $q$ and $L$
with the following properties:

\begin{enumerate}
\item $q\leq p\upharpoonright\alpha.$

\item $q\in\mathbb{P}_{\alpha}^{K_{1}}.$

\item If $\beta\in A\setminus\left\{  \alpha\right\}  ,$ then $q\left(
\beta\right)  =p\left(  \beta\right)  $.

\item $L\in\left[  B\right]  ^{<\omega}$ is an antichain.

\item For every $q^{\prime}\leq q,$ if $q^{\prime}$ follows $K_{1},$ then
$q_{1}$ is compatible with an element of $L.$
\end{enumerate}

We now define $L_{1}=\{r^{\frown}\dot{x}_{i}^{r}\mid r\in L\wedge\dot{x}%
_{i}^{r}\in Y^{r}\},$ note that $L_{1}$ is a finite antichain of $D.$ Define
$\overline{q}=q^{\frown}p\left(  \alpha\right)  ,$ we claim that $\overline
{q}$ and $L_{1}$ have the desired properties. Clearly $\overline{q}%
\in\mathbb{P}_{\alpha+1}^{K}.$ Now, let $q_{1}\leq\overline{q}$ that follows
$K.$ Since $q_{1}\upharpoonright\alpha\leq\overline{q}\upharpoonright\alpha=q$
and $q_{1}\upharpoonright\alpha$ follows $K_{1},$ we know that there is $r\in
L$ compatible with $q_{1}\upharpoonright\alpha.$ Let $q_{2}\leq$
$q_{1}\upharpoonright\alpha,r$ and note that $q_{2}\Vdash``\dot{Z}=\{\dot
{x}_{i}^{r}[\dot{G}]\mid i<k\}\mrq,$ hence (without lost of
generality), there is $i$ such that $q_{2}$ forces that $q_{1}\left(
\alpha\right)  $ and $\dot{x}_{i}^{r}$ are compatible (recall that
$q_{1}\left(  \alpha\right)  $ is forced to be of the form $(s,n,\dot{J})$
since $q_{1}$ follows $K$). It follows that $q_{1}$ and $r^{\frown}\dot{x}%
_{i}^{r}$ are compatible.

Finally, we consider the case when $\alpha$ is a limit ordinal and the
proposition is true for every $\beta<\alpha.$ This case is similar to the one
where $\alpha\notin A.$ We first find $\beta<\alpha$ such that $A,\dom\left(
p\right)  \subseteq\beta.$ Define $\overline{D}$ as the set of all
$q\in\mathbb{P}_{\beta}$ such that there is $\overline{q}\in\mathbb{P}%
_{\alpha}$ with the following properties:

\begin{enumerate}
\item $\overline{q}\upharpoonright\beta=q.$

\item $\overline{q}\in D.$

\item If $\xi\in\left(  \dom\left(  \overline{q}\right)  \setminus\beta\right)
\cap\left(  H\cup E\right)  $ and $\bar q\upharpoonright \xi \Vdash \overline{q}\left(  \xi\right)  =(\dot
{z},\dot{m},\dot{J})$ then $\overline{q}\upharpoonright\xi\Vdash``4\left\vert
\dot{J}\right\vert \leq\dot{m}\mrq.$

\item $\overline{q}\upharpoonright\lbrack\beta,\alpha)$ has the decreasing condition.

\item There is $n_{\overline{q}}$ such that for every $\xi\in \dom\left(
\overline{q}\right)  \setminus\beta,$ the condition $\overline{q}%
\upharpoonright\xi\Vdash``len\left(  \overline{q}\left(  \xi\right)  \right)
\leq n_{\overline{q}}\mrq.$
\end{enumerate}

It is easy to see that $\overline{D}$ is an open dense subset of
$\mathbb{P}_{\beta}$ (it is dense by Lemma \ref{pure descending}). By the
induction hypothesis, there are $q\leq p$ following $K$ and an antichain
$L=\left\{  q_{i}\mid i<k\right\}  $ $\subseteq\overline{D}$ such that for
every $r\leq q$ that follows $K,$ $r$ is compatible with an element of $L.$
For every $i<k,$ choose $\overline{q}_{i}\in D$ witnessing that $q_{i}%
\in\overline{D}.$ Find $n\in\omega$ such that $n>n_{\overline{q}_{i}}$ for
every $q_{i}\in L.$ By Lemma \ref{crecer stems}, we may assume that all of the
stems in $\dom\left(  q\right)  \setminus A$ are forced to be larger than $n.$
Let $B_{i}=\dom(\overline{q}_{i})$ for every $i<k.$ We now define a condition
$\widehat{q}\in\mathbb{P}_{\alpha}$ with the following properties:

\begin{enumerate}
\item $\widehat{q}\upharpoonright\beta=q.$

\item $\dom(\widehat{q})=\dom\left(  q\right)  \cup%
{\textstyle\bigcup\limits_{i<k}}
B_{i}$

\item For every $i<k$ and $\xi\in B_{i},$ we have that $q_{i}\upharpoonright
\xi\Vdash``\widehat{q}\left(  \xi\right)  =\overline{q}_{i}\left(  \xi\right)
\mrq.$

\item For every $i<k$ and $\xi\in\alpha$ such that $\xi\notin\beta\cup B_{i},$
we have that $q_{i}\upharpoonright\xi\Vdash``\widehat{q}\left(  \xi\right)
=1_{\mathbb{\dot{Q}}_{\xi}}\mrq$ (where $1_{\mathbb{\dot{Q}%
}_{\xi}}$ is the name of the largest condition).

\item If $r\in\mathbb{P}_{\beta}$ is incompatible with every $q_{i}\in L$ and
$\xi\in%
{\textstyle\bigcup\limits_{i<k}}
B_{i},$ then $r\upharpoonright\xi\Vdash``\widehat{q}\left(  \xi\right)
=1_{\mathbb{\dot{Q}}_{\xi}}\mrq$
\end{enumerate}

Let $L_{1}=\{\overline{q}_{i}\mid i<k\},$ we will show that $\widehat{q}$ and
$L_{1}$ have the desired properties. It is easy to see that $\widehat{q}%
\in\mathbb{P}_{\alpha}^{K}.$ Now, let $r\leq\widehat{q}$ that follows $K.$
Clearly, $r\upharpoonright\beta$ extends $q$ and follows $K,$ so there is
$i<k$ such that $q_{i}$ is compatible with $r.$ It is easy to see that $r$ is
compatible with $\overline{q}_{i}.$
\end{proof}

\bigskip

We can now prove the following:

\begin{proposition}
There is a model of \textsf{ZFC }such that the following conditions hold:

\begin{enumerate}
\item $\mathfrak{c}=\omega_{3}.$

\item $\mathfrak{ie}=\omega_{2}.$

\item There are families $\{\mathcal{A}^{\gamma}\mid\gamma\in\omega_{1}\},$
$\mathcal{B}=\left\{  f_{\alpha}\mid\alpha\in\omega_{2}\right\}  $ such that
the following holds:

\begin{enumerate}
\item $\mathcal{A}^{\gamma}\subseteq\left[  \omega\right]  ^{\omega}$ is a
\textsf{MAD} family of size $\omega_{2}$ (for every $\gamma\in\omega_{1}$).

\item $\mathcal{B}\subseteq PFun\left(  \Delta\right)  $ is a \textsf{MAD} family.

\item Let $\pi:PFun\left(  \Delta\right)  \longrightarrow\left[
\omega\right]  ^{\omega}$ be the function defined by $\pi\left(  f\right)
=\dom\left(  f\right)  .$ We have that $\pi\upharpoonright\mathcal{B}%
:\mathcal{B\longrightarrow}%
{\textstyle\bigcup\limits_{\gamma\in\omega_{1}}}
\mathcal{A}_{\gamma}$ is bijective.
\end{enumerate}
\end{enumerate}
\end{proposition}

\begin{proof}
We start with a ground model such that $V\models\mathfrak{c}=\omega_{3}$ and
we will force with $\mathbb{P}_{\omega_{2}}.$ Let $G\subseteq\mathbb{P}%
_{\omega_{2}}$ be a generic filter. It is easy to see that $V\left[  G\right]
\models\mathfrak{c}=\omega_{3}.$ For every $\gamma\in\omega_{1},$ let
$\mathcal{A}^{\gamma}=\left\{  A_{\gamma}^{\alpha}\mid\alpha\in D_{\gamma
}\right\}  $. We have the following:

\begin{claim}
Let $\gamma\in\omega_{1}.$

\begin{enumerate}
\item $\mathcal{A}^{\gamma}\subseteq\left[  \omega\right]  ^{\omega}$ is a
\textsf{MAD} family of size $\omega_{2}.$

\item For every $X\in V\left[  G\right]  ,$ if $X\in\mathcal{I}\left(
\mathcal{A}^{\gamma}\right)  ^{+},$ then the set $\{\alpha\in D_{\gamma}%
\mid\left\vert X\cap A_{\gamma}^{\alpha}\right\vert =\omega\}$ has size
$\omega_{2}.$
\end{enumerate}
\end{claim}

The claim follows easily by Lemma \ref{Basic Mathias}. A more interesting fact
is the following:

\begin{claim}
$V\left[  G\right]  \models%
{\textstyle\bigcap\limits_{\gamma\in\omega_{1}}}
\mathcal{I}\left(  A^{\gamma}\right)  = [\omega]^{<\omega}.$
\end{claim}

The ideas for the proof the claim were inspired on the several
``copying'' arguments used in the method of
forcing with side conditions. The reader may consult \cite{PartitionProblems}
and \cite{NotesonForcingAxioms} to learn more about forcing with side conditions.

Let $\dot{X}$ be a $\mathbb{P}_{\omega_{2}}$-name for an infinite subset of
$\omega.$ Let $M\in V$ be a countable elementary submodel of \textsf{H}%
$(\left(  2^{\omega_{3}}\right)  ^{+})$ such that $\dot{X},\mathbb{P}%
_{\omega_{2}}\in M.$ Choose $\gamma\in\omega_{1}\setminus M,$ we will show
that $\dot{X}$ is forced to be in $\mathcal{I}\left(  \mathcal{A}^{\gamma
}\right)  ^{+}.$ In fact, we will prove that $\dot{X}$ will have infinite
intersection with every element of $\mathcal{A}^{\gamma}.$ Note that
$D_{\gamma}\cap M=\emptyset$ since $\gamma\notin M$ (recall that $\left\{
D_{\eta}\mid\eta\in\omega_{1}\right\}  \in M$ since $\mathbb{P}_{\omega_{2}%
}\in M$).

Let $\xi\in D_{\gamma},$ $k\in\omega$ and $p\in\mathbb{P}_{\omega_{2}}$ (in
general, $p\notin M$). We must find an extension of $p$ forcing that $\dot{X}$
and $A_{\gamma}^{\xi}$ intersect beyond $k.$ We may assume that $\xi\in
\dom\left(  p\right)  ,$ $p$ is pure and has the descending condition. Let $n$
be the height of $p.$ We may also assume that $n>k.$ For technical reasons,
assume that $0\in \dom\left(  p\right)  .$ Let $B=\dom\left(  p\right)  \cap M$
and $A=B\cap E.$ We now find $K:A\longrightarrow\omega^{<\omega}$ such that
$p$ follows $K.$ Note that $p\in\mathbb{P}_{\alpha}^{K}.$ Let $\dom\left(
p\right)  =\left\{  \alpha_{0},...,\alpha_{m}\right\}  $ where $\alpha
_{i}<\alpha_{j}$ whenever $i<j$. 

\begin{claim}
There is $\overline{p}\in M\cap\mathbb{P}_{\omega_{2}}$ such that for every
$i\leq m,$ the following holds:

\begin{enumerate}
\item $\overline{p}$ is pure of height $n.$

\item $\dom\left(  \overline{p}\right)  =\left\{  \delta_{0},...,\delta
_{m}\right\}  $ (where $\delta_{i}<\delta_{j}$ whenever $i<j$) and $B\subseteq
\dom\left(  \overline{p}\right)  .$

\item $\overline{p}\in\mathbb{P}_{\omega_{2}}^{K}.$

\item If $\alpha_{i}\in B,$ then $\delta_{i}=\alpha_{i}$.

\item If $\alpha_{i}\notin B,$ then $\delta_{i}<\alpha_{i}$.

\item $\alpha_{i}\in E$ if and only if $\delta_{i}\in E.$

\item $\alpha_{i}\in H$ if and only if $\delta_{i}\in H.$

\item For every $\eta\in M\cap\omega_{1},$ if $\alpha_{i}\in D_{\eta}$ then
$\delta_{i}\in D_{\eta}.$

\item For every $j\leq m,$ if $\alpha_{i},\alpha_{j}\in%
{\textstyle\bigcup\limits_{\eta\in\omega_{1}}}
D_{\eta}$ then $\alpha_{i},\alpha_{j}$ are in the same element of the
partition if and only if $\delta_{i},\delta_{j}$ are in the same element of
the partition.

\item If $\alpha_{i}\in H,$ then the following holds:

\begin{enumerate}
\item If $p\left(  \alpha_{i}\right)  =\left(  s_{\alpha_{i}},n,\{\dot f_\mu: \mu \in J_{\alpha_{i}}^{p}\}\right)  ,$ then $\overline{p}\left(  \delta_{i}\right)  =\left(
s_{\alpha_{i}},n,\{\dot f_\mu: \mu \in J_{\delta_{i}}^{\overline{p}}\}\right)  $ (i.e. the stem of
$p\left(  \alpha_{i}\right)  $ and $\overline{p}\left(  \delta_{i}\right)  $
is the same).

\item For every $j<i,$ we have that $\alpha_{j}\in J_{\alpha_{i}%
}^{p}$ if and only if $\delta_{j}\in J_{\delta_{i}}^{\overline{p}}.$
\end{enumerate}

\item If $\alpha_{i}\in
D_{\eta}$ for some $\eta<\omega_1$, then the following holds:

\begin{enumerate}
\item If $p\upharpoonright\alpha_{i}\Vdash``p\left(  \alpha_{i}\right)
=\left(  s_{\alpha_{i}},\{\dot A_\eta^\mu: \mu \in J_{\alpha_{i}}^{p}\}\right)  \mrq,$ then
$\overline{p}\upharpoonright\delta_{i}\Vdash``\overline{p}\left(  \delta
_{i}\right)  =\left(  s_{\alpha_{i}},\{\dot A_\eta^\mu: \mu \in J_{\delta_{i}}^{\overline{p}}\}\right)
\mrq$ (i.e. the stem of $p\left(  \alpha_{i}\right)  $ and
$\overline{p}\left(  \delta_{i}\right)  $ is the same).

\item For every $j<i,$ we have that $\alpha_{j}\in J_{\alpha_{i}%
}^{p}$ if and only if $\delta_{j}\in J_{\delta_{i}%
}^{\overline{p}}$ (where $\alpha_{i}\in D_{\eta}$ and $\delta_{i}\in
D_{\eta^{\prime}}$).
\end{enumerate}

\item If $\alpha_{i}\in E,$ then the following holds:

\begin{enumerate}
\item If $p\left(  \alpha_{i}\right)  =\left(  s_{\alpha_{i}},m_{\alpha_{i}%
},J_{\alpha_{i}}^{p}\right)  ,$ then $\overline{p}\left(  \delta_{i}\right)
=\left(  s_{\alpha_{i}}m_{\alpha_{i}},J_{\delta_{i}}^{\overline{p}}\right)  $
(i.e. the stems of $p\left(  \alpha_{i}\right)  $ and $\overline{p}\left(
\delta_{i}\right)  $ are the same).

\item If $\alpha_{i}\in M,$ then there exists names $\tau_0, \dots, \tau_{k_{\alpha_i}-1}$ such that $J_{\alpha_i}^{\bar p}=\{(\tau_0, \mathbbm 1), \dots, (\tau_{k_{\alpha_i}-1}, \mathbbm 1)\}$ (recall that in this case, $\alpha_{i}=\delta_{i}$).
\end{enumerate}
\end{enumerate}
\end{claim}

The claim is almost an immediate consequence of the elementarity of $M,$ point
5 is the only one that requires being slightly more careful. For every
$\alpha_{i}\notin B,$ we define the following:

\begin{enumerate}
\item $\xi_{i}^{0}=max\left(  B\right)  \cap\alpha_{i}$ (this is well defined
since $0\in B$).

\item $\xi_{i}^{1}=min\left(  M\cap\left(  \omega_{2}+1\right)  \setminus
\alpha_{i}\right)  .$
\end{enumerate}

Note that $\xi_{i}^{0},\xi_{i}^{1}\in M$ and $\xi_{i}^{0}<\alpha_{i}<\xi
_{i}^{1}.$ The claim then follows by applying elementarity and requiring that
$\xi_{i}^{0}<\delta_{i}<\xi_{i}^{1}.$ Since $\delta_{i}\in M$ and is smaller
that $\xi_{i}^{1},$ it follows that $\delta_{i}<\alpha_{i}.$

Let $\overline{p}$ be as in the claim. We now define
$$D=\left\{
r\in\mathbb{P}_{\omega_{2}}\mid\exists l_{r}\in\omega\left(  r\Vdash
``l_{r}=min\left(  \dot{X}\setminus n\right)  \mrq\right)
\right\}  .$$ 
Clearly $D\subseteq\mathbb{P}_{\omega_{2}}$ is an open dense
subset and $D\in M.$ Since $\overline{p}\in\mathbb{P}_{\omega_{2}}^{K},$
applying Lemma \ref{Compacidad iteracion}, there is $q\leq\overline{p}$ as in
the lemma. We may even assume that $q\in M.$ Note that in general, $q$ might
not be pure (we could extend it to a pure condition, but it might not follow
$K$ anymore). Let $L\in\left[  D\right]  ^{<\omega}$ such that for every
$r\leq q$, if $r$ follows $K,$ then $r$ is compatible with an element of $L.$
Let $Z=\left\{  l_{r}\mid r\in L\right\}  $ and note that $Z\cap n=\emptyset.$
It is clear that if $r\in L,$ then $r\Vdash``Z\cap\dot{X}\neq\emptyset
\mrq.$ Let $n_{1}=max\left(  Z\right)  +1.$

We now define the condition $p_{Z}$ with the following properties:

\begin{enumerate}
\item $\dom\left(  p_{Z}\right)  =\dom\left(  p\right)  .$

\item For every $\eta\in \dom\left(  p_{Z}\right)  ,$ the following holds:

\begin{enumerate}
\item If $\eta\notin D_{\gamma},$ then $p_{Z}\left(  \eta\right)  =p\left(
\eta\right)  .$

\item Let $\eta\in D_{\gamma}$ with $\eta\neq\xi.$ If $p\left(  \eta\right)
=\left(  s_{\eta}^{p},\{\dot A_\gamma^\mu: \mu \in J_{\eta}^{p}\}\right)  $ define $s_{\eta}^{p_{Z}}%
:n_{1}\longrightarrow2$ such that $s_{\eta}^{p}\subseteq s_{\eta}^{p_{Z}}$ and
$s_{\eta}^{p_{Z}}\left(  i\right)  =0$ for every $i\in\lbrack n,n_{1}).$ Let
$p_{Z}\left(  \eta\right)  =\left(  s_{\eta}^{p_{Z}},\{\dot A_\gamma^\mu: \mu \in J_{\eta}^{p}\}\right)  .$

\item If $p\left(  \xi\right)  =\left(  s_{\xi}^{p},\{\dot A_\gamma^\mu: \mu \in J_{\xi}^{p}\}\right)  $
define $s_{\xi}^{p_{Z}}:n_{1}\longrightarrow2$ such that $s_{\xi}^{p}\subseteq
s_{\xi}^{p_{Z}}$ and $s_{\xi}^{p_{Z}}\left(  i\right)  =1$ for every
$i\in\lbrack n,n_{1}).$ Let $p_{Z}\left(  \xi\right)  =\left(  s_{\xi
}^{p_{Z}},\{\dot A_\gamma^\mu: \mu \in J_{\xi}^{p}\}\right)  .$
\end{enumerate}
\end{enumerate}

Note that $p_{Z}\Vdash``Z\subseteq A_{\gamma}^{\xi}\mrq$. Since $J^p_\xi\subseteq \dom (p\upharpoonright\xi)$, it is
follows from (b) that $p_{Z}\leq p$. We now define the condition $r$ as follows:

\begin{enumerate}
\item $\dom\left(  r\right)  =\dom\left(  p_{Z}\right)  \cup \dom\left(
q\right)  .$

\item If $\eta\in \dom\left(  q\right)  \setminus \dom\left(  p_{Z}\right)  ,$
then $r\left(  \eta\right)  =q\left(  \eta\right)  .$

\item Let $\eta\in \dom\left(  p_{Z}\right)  ,$ so $\eta=\alpha_{i}$ for some
$i\leq m.$ We have the following:

\begin{enumerate}
\item Assume $\alpha_{i}\in D_{\gamma^{\prime}}$ with $\gamma^{\prime}\notin
M$ (so $\eta\notin \dom\left(  q\right)  $), define $r\left(  \alpha
_{i}\right)  =p_Z\left(  \alpha_{i}\right)  $ (note that this will be
the case when $\gamma^{\prime}=\gamma$).

\item Assume $\alpha_{i}\in D_{\gamma^{\prime}}$ with $\gamma^{\prime}\in M$. Let
$p_{Z}\left(  \alpha_{i}\right)  =\left(  s_{\alpha_{i}}^{p_{Z}},\{\dot A_{\gamma'}^\mu: \mu \in J_{\alpha
_{i}}^{p_{Z}}\}\right)  $ and $q\upharpoonright \delta_i \Vdash q\left(  \delta_{i}\right)  =\left(  \dot
{t}_{\delta_{i}}^{q},\{\dot A_{\gamma'}^\mu: \mu \in \dot{J}_{\delta_{i}}^{q}\}\right)  $ (since $q$ is not
pure, $\dot{t}_{\delta_{i}}^{q}$ and $\dot{J}_{\delta_{i}}^{q}$ might be names
and not actual objects). Define $r\left(  \alpha_{i}\right)  =\left(  \dot
{t}_{\delta_{i}}^{q},\{\dot A_{\gamma'}^\mu: \mu \in J_{\alpha_{i}}^{p_{Z}}\cup\dot{J}_{\delta_{i}}%
^{q}\}\right)  $.  In here, note that $\dot{t}_{\delta_{i}}^{q}$
is a $\mathbb{P}_{\delta_{i}}$-name, since $\delta_{i}\leq\alpha_{i}$ it is also
a $\mathbb{P}_{\alpha_{i}}$-name, so the definition at least makes sense.

\item Assume $\alpha_{i}\in H$. Let $p_{Z}\left(  \alpha_{i}\right)  =\left(
s_{\alpha_{i}}^{p_{Z}},n,\{\dot f_{\mu}: \mu \in J_{\alpha_{i}}^{p_{Z}}\}\right)  $, $q\left(
\delta_{i}\right)  =\left(  \dot{t}_{\delta_{i}}^{q},\dot{m}_{\delta_{i}}%
^{q},\{\dot f_{\mu}:\mu \in\dot{J}_{\delta_{i}}^{q}\}\right)  $, and $r\left(  \alpha_{i}\right)
=\left(  \dot{t}_{\delta_{i}}^{q},\dot{m}_{\delta_{i}}^{q},\{\dot f_{\mu}: \mu \in\dot{J}_{\delta
_{i}}^{q}\cup J_{\alpha_{i}}^{p_{Z}}\}\right)  .$

\item Assume $\alpha_{i}\in E$ and $\alpha_{i}\notin \dom\left(  q\right)  .$
Define $r\left(  \alpha_{i}\right)  =p_{Z}\left(  \alpha_{i}\right)  .$

\item Assume $\alpha_{i}\in E$ and $\alpha_{i}\in \dom\left(  q\right)  $ (so
$\delta_{i}=\alpha_{i}$ and $\alpha_{i}\in A$). Let $p_{Z}\left(  \alpha
_{i}\right)  =\left(  s_{\alpha_{i}}^{p_{Z}},n,\dot{J}_{\alpha_{i}}^{p_{Z}%
}\right)  $ and note that in here we have that $q\left(  \delta_{i}\right)
=\left(  s_{\alpha_{i}}^{p_{Z}},n,\dot{J}_{\delta_{i}}^{q}\right)  $ (this is
because $\alpha_{i}\in A,$ so $q\left(  \delta_{i}\right)  =\overline
{p}\left(  \delta_{i}\right)  $). Define $r\left(  \delta_{i}\right)  =\left(
s_{\alpha_{i}}^{p_{Z}},n,\dot{J}_{\delta_{i}}^{q}\cup\dot{J}_{\alpha_{i}%
}^{p_{Z}}\right)  .$
\end{enumerate}
\end{enumerate}

It might not be immediately obvious that $r$ is a condition, since the
``size requirement'' may fail in the
coordinates of $E$ or $H.$ We will show that this is not the case.

\begin{claim}
Let $\eta\in \dom\left(  r\right)  .$

\begin{enumerate}
\item $r\upharpoonright\eta\in\mathbb{P}_{\eta}.$

\item $r\upharpoonright\eta\Vdash``r\left(  \eta\right)  \in\mathbb{\dot{Q}%
}_{\eta}\mrq.$

\item $r\upharpoonright\eta\leq q\upharpoonright\eta.$

\item $r\upharpoonright\eta\Vdash``r\left(  \eta\right)  \leq q\left(
\eta\right)  \mrq.$
\end{enumerate}

\end{claim}

We will prove the claim. Note that points 3 and 4 are trivial once we know
that $r\upharpoonright\eta$ is a condition. We proceed by induction, it is
enough to show that if $r\upharpoonright\eta\in\mathbb{P}_{\eta}$ and
$r\upharpoonright\eta\leq q\upharpoonright\eta,$ then $r\upharpoonright
\eta\Vdash``r\left(  \eta\right)  \in\mathbb{\dot{Q}}_{\eta}\mrq%
.$ Furthermore, this is clear whenever $\eta\in \dom\left(  q\right)  \setminus
\dom\left(  p_{Z}\right)  ,$ $\eta\notin H\cup E$ or $\eta\in E\setminus
\dom\left(  q\right)  .$ We focus on the other cases. From now on, $\eta\in
\dom\left(  p_{Z}\right)  ,$ so we may assume that $\eta=\alpha_{i}$ for some
$i\leq m.$

\begin{case}$\alpha_{i}\in H$.
\end{case}
In here, $p_{Z}\left(  \alpha_{i}\right)  =\left(
s_{\alpha_{i}}^{p},n,\{\dot f_{\mu}: \mu \in J_{\alpha_{i}}^{p}\}\right)  $, $q\left(
\delta_{i}\right)  =\left(  \dot{t}_{\delta_{i}}^{q},\dot{m}_{\delta_{i}}%
^{q},\{\dot f_{\mu}:\mu \in\dot{J}_{\delta_{i}}^{q}\}\right)  $, and $r\left(  \alpha_{i}\right)
=\left(  \dot{t}_{\delta_{i}}^{q},\dot{m}_{\delta_{i}}^{q},\{\dot f_{\mu}: \mu \in\dot{J}_{\delta
_{i}}^{q}\cup J_{\alpha_{i}}^{p}\}\right)$. As $\overline{p}\left(  \delta_{i}\right)  =\left(
s_{\alpha_{i}}^{p},n,\{\dot f_{\mu}: \mu \in J_{\delta_{i}}^{\bar p}\}\right)  $ and since $q\leq\overline{p},$ we get
that $q\upharpoonright\delta_{i}\Vdash``n\leq\dot{m}_{\delta_{i}}%
^{q}\mrq.$ Furthermore, $q\upharpoonright\delta_{i}%
\Vdash``4\left\vert \dot{J}_{\delta_{i}}^{q}\right\vert \leq\dot{m}_{\delta_{i}%
}^{q}\mrq$. We also know that $4\left\vert J_{\alpha_{i}}%
^{p}\right\vert \leq n,$ (since $p$ is pure), hence
$q\upharpoonright\delta_{i}\Vdash_{\mathbb P_{\delta_i}}``4\left\vert \dot{J}_{\delta_{i}}%
^{q}\right\vert ,4\left\vert J_{\alpha_{i}}^{p_{Z}}\right\vert \leq \dot
{m}_{\delta_{i}}^{q}\mrq.$ Since $r\upharpoonright\alpha_{i}\leq r\upharpoonright\delta_{i}\leq
q\upharpoonright\delta_{i}$, $\mathbb P_{\delta_i}$ is completely embedded into $\mathbb P_{\alpha_i}$ and the formula is absolute for transitive models of ZFC, we get that $r\upharpoonright\alpha_{i}%
\Vdash_{\mathbb P_{\alpha_i}}``4\left\vert \dot{J}_{\delta_{i}}^{q}\right\vert ,4\left\vert
J_{\alpha_{i}}^{p_{Z}}\right\vert \leq\dot{m}_{\delta_{i}}^{q}\mrq%
,$ so $r\left(  \alpha\right)  $ is forced to be a condition by Lemma
\ref{lemma del 4}.
\begin{case}
$\alpha_i \in \dom q\cap E$.
\end{case}
In here, $p_{Z}\left(  \alpha_{i}\right)  =\left(
s_{\alpha_{i}}^{p},m_{\alpha_i},\dot J_{\alpha_{i}}^{p}\right)  $, $q\left(  \alpha_{i}\right)  =\overline{p}(\alpha_i)=\left(
s_{\alpha_{i}}^{p},m_{\alpha_i},\dot J_{\alpha_{i}}^{q}\right)  $ and $r(\alpha_i)=\left(
s_{\alpha_{i}}^{p},m_{\alpha_i},\dot J_{\alpha_{i}}^{q}\cup \dot J_{\alpha_{i}}^{p}\right)$. Clearly, $r\upharpoonright \alpha_i \Vdash 
4\left|\dot J_{\alpha_i}^q\right|, 4\left|\dot J_{\alpha_i}^p\right|\leq m_\alpha$ since any condition forces this statement, so $r\left(  \alpha\right)  $ is forced to be a condition by Lemma \ref{lemma del 4}.\\

We now know that $r$ is indeed a condition and that $r\leq q.$ Note that $r$
follows $K$.

We will now prove that $r\leq p_{Z}.$ Let $\alpha_{i}\in
\dom\left(  p_{Z}\right)  ,$ assume that we know that $r\upharpoonright
\alpha_{i}\leq p_{Z}\upharpoonright\alpha_{i},$ we will prove that
$r\upharpoonright\alpha_{i}\Vdash``r\left(  \alpha_{i}\right)  <p_{Z}\left(
\alpha_{i}\right)  \mrq.$ We proceed by cases:
\begin{claim}$r\leq p_{Z}.$
\end{claim}
\begin{case}
$\alpha_{i}\in D_{\gamma^{\prime}}$ with $\gamma^{\prime}\notin M.$
\end{case}

This case is immediate by the definition.

\begin{case}
$\alpha_{i}\in D_{\gamma^{\prime}}$ with $\gamma^{\prime}\in M$ and
$\alpha_{i}\in \dom\left(  q\right)  $ (hence $\delta_{i}=\alpha_{i}$).
\end{case}

In here, we have that 
$$p_{Z}\left(  \alpha_{i}\right)  =\left(  s_{\alpha_{i}%
}^{p_{Z}},\{\dot A_{\gamma'}^\mu: \mu \in J_{\alpha_{i}}^{p_{Z}}\}\right)  , \ q\left(  \alpha_{i}\right)
=\left(  \dot{t}_{\alpha_{i}}^{q},\{\dot A_{\gamma'}^\mu: \mu \in\dot{J}_{\alpha_{i}}^{q}\}\right)  $$
and
$r\left(  \alpha_{i}\right)  =\left(  \dot{t}_{\alpha_{i}}^{q},\{\dot A_{\gamma'}^\mu: \mu \in J_{\alpha_{i}%
}^{p_{Z}}\cup\dot{J}_{\alpha_{i}}^{q}\}\right)  $. Since $r\leq q,$ we have that
$r\leq\overline{p},$ so $r\upharpoonright\alpha_{i}\Vdash``s_{\alpha_{i}%
}^{p_{Z}}\subseteq\dot{t}_{\alpha_{i}}^{q}\mrq$ (in this case, $s_{\alpha_i}^{\bar p}=s_{\alpha_i}^p=s_{\alpha_i}^{p_Z}$).

Now, let $\alpha_{j}\in J_{\alpha_{i}}^{p_{Z}}$ (recall that the stem of
$r\left(  \alpha_{j}\right)  $ is $\dot{t}_{\alpha_{j}}^{q}$). We need to
prove that $r\upharpoonright\alpha_{i}\Vdash``\left(  \dot{t}_{\alpha_{i}}%
^{q}\right)  ^{-1}\left(  1\right)  \cap A_{\gamma^{\prime}}^{\alpha_{j}%
}\subseteq n\mrq.$ Let $\dot{m}_{\alpha_{i}},\dot{m}_{\alpha
_{j}}$ such that $q\upharpoonright\alpha_{i}\Vdash``\dot{t}_{\alpha_{i}}%
^{q}:\dot{m}_{\alpha_{i}}\longrightarrow2\mrq$ and
$q\upharpoonright\alpha_{j}\Vdash``\dot{t}_{\alpha_{j}}^{q}:\dot{m}%
_{\alpha_{j}}\longrightarrow2\mrq.$ Since $q$ satisfies the
$K$-descending condition, we know that $q\upharpoonright\alpha_{i}\Vdash
``\dot{m}_{\alpha_{j}}\geq\dot{m}_{\alpha_{i}}\mrq.$ Since
$q\upharpoonright\alpha_{j}\Vdash``A_{\gamma^{\prime}}^{\alpha_{j}}\cap\dot
{m}_{\alpha_{j}}=\left(  \dot{t}_{\alpha_{j}}^{q}\right)  ^{-1}\left(
1\right)  \mrq,$ we get that $q\upharpoonright\alpha_{i}%
\Vdash``A_{\gamma^{\prime}}^{\alpha_{j}}\cap\dot{m}_{\alpha_{i}}=\left(
\dot{t}_{\alpha_{j}}^{q}\right)  ^{-1}\left(  1\right)  \cap\dot{m}%
_{\alpha_{i}}\mrq.$ Since $r\leq\overline{p},$ we know that
$r\Vdash``A_{\gamma^{\prime}}^{\alpha_{i}}\cap A_{\gamma^{\prime}}^{\alpha
_{j}}\subseteq n\mrq.$ In particular, $\dot{t}_{\alpha_{i}}^{q}$
is forced to be disjoint with $\dot{A}_{\gamma^{\prime}}^{\alpha_{j}}\setminus
n,$ so we get that $r\Vdash``\left(  \dot{t}_{\alpha_{i}}^{q}\right)
^{-1}\left(  1\right)  \cap\left(  \dot{t}_{\alpha_{j}}^{q}\right)
^{-1}\left(  1\right)  \subseteq n\mrq,$ hence $r\upharpoonright
\alpha_{i}\Vdash``\left(  \dot{t}_{\alpha_{i}}^{q}\right)  ^{-1}\left(
1\right)  \cap A_{\gamma^{\prime}}^{\alpha_{j}}\subseteq n\mrq,$
which is what we wanted to prove.

\begin{case}
$\alpha_{i}\in D_{\gamma^{\prime}}$ with $\gamma^{\prime}\in M$ and
$\alpha_{i}\notin \dom\left(  q\right)  $ (so $\delta_{i}<\alpha_{i}$).
\end{case}

In here, we have that $p_{Z}\left(  \alpha_{i}\right)  =\left(  s_{\alpha_{i}%
}^{p_{Z}},\{\dot A_{\gamma}^\mu: \mu \in J_{\alpha_i}^{p_Z}\}\right)  $, $q\left(  \delta_{i}\right)
=\left(  \dot{t}_{\delta_{i}}^{q},\{A_{\gamma'}^\mu: \mu \in \dot{J}_{\delta_{i}}^{q}\}\right)  $ and
$r\left(  \alpha_{i}\right)  =\left(  \dot{t}_{\delta_{i}}^{q},\{A_{\gamma'}^\mu: \mu \in \dot J_{\delta_i}^q\cup J_{\alpha_i}^{p_z}\}\right)  $. Since $r\leq q,$ we have
that $r\leq\overline{p},$ so $r\upharpoonright\delta_{i}\Vdash``s_{\alpha_{i}%
}^{p_{Z}}\subseteq\dot{t}_{\delta_{i}}^{q}\mrq$ (recall that
$s_{\alpha_{i}}^{p_{Z}}$ is the stem of $\overline{p}\left(  \delta
_{i}\right)  $).

Now, let $\alpha_{j}\in J_{\alpha_{i}}^{p_{Z}}$ (recall that the stem of
$r\left(  \alpha_{j}\right)  $ is $\dot{t}_{\delta_{j}}^{q}$). We need to
prove that $r\upharpoonright\alpha_{i}\Vdash``\left(  \dot{t}_{\delta_{i}}%
^{q}\right)  ^{-1}\left(  1\right)  \cap A_{\gamma^{\prime}}^{\alpha_{j}%
}\subseteq n\mrq.$ Let $\dot{m}_{\delta_{i}},\dot{m}_{\delta
_{j}}$ such that $q\upharpoonright\delta_{i}\Vdash``\dot{t}_{\delta_{i}}%
^{q}:\dot{m}_{\delta_{i}}\longrightarrow2\mrq$ and
$q\upharpoonright\delta_{j}\Vdash``\dot{t}_{\delta_{j}}^{q}:\dot{m}%
_{\delta_{j}}\longrightarrow2\mrq.$ Since $q$ satisfies the
$K$-descending condition, we know that $q\upharpoonright\delta_{i}\Vdash
``\dot{m}_{\delta_{j}}\geq\dot{m}_{\delta_{i}}\mrq.$ Since
$q\upharpoonright\delta_{j}\Vdash``A_{\gamma^{\prime}}^{\delta_{j}}\cap\dot
{m}_{\delta_{j}}=\left(  \dot{t}_{\delta_{j}}^{q}\right)  ^{-1}\left(
1\right)  \mrq,$ we get that $q\upharpoonright\delta_{i}%
\Vdash``A_{\gamma^{\prime}}^{\delta_{j}}\cap\dot{m}_{\delta_{i}}=\left(
\dot{t}_{\delta_{j}}^{q}\right)  ^{-1}\left(  1\right)  \cap\dot{m}%
_{\delta_{i}}\mrq.$ Since $r\leq\overline{p},$ we know that
$r\Vdash``A_{\gamma^{\prime}}^{\delta_{i}}\cap A_{\gamma^{\prime}}^{\delta
_{j}}\subseteq n\mrq.$ In particular, $\dot{t}_{\delta_{i}}^{q}$
is forced to be disjoint with $\dot{A}_{\gamma^{\prime}}^{\delta_{j}}\setminus
n,$ so we get that $r\Vdash``\left(  \dot{t}_{\alpha_{i}}^{q}\right)
^{-1}\left(  1\right)  \cap\left(  \dot{t}_{\delta_{j}}^{q}\right)
^{-1}\left(  1\right)  \subseteq n\mrq,$ hence $r\upharpoonright
\alpha_{i}\Vdash``\left(  \dot{t}_{\delta_{i}}^{q}\right)  ^{-1}\left(
1\right)  \cap A_{\gamma^{\prime}}^{\alpha_{j}}\subseteq n\mrq,$
which is what we wanted to prove.

\begin{case}
$\alpha_{i}\in H$ and $\alpha_{i}\in \dom\left(  q\right)  $ (so $\alpha
_{i}=\delta_{i}$).
\end{case}

In here, we have that $p_{Z}\left(  \alpha_{i}\right)  =\left(  s_{\alpha_{i}%
}^{p_{Z}},n,\{\dot f_\mu: \mu \in J_{\alpha_{i}}^{p_{Z}}\}\right)  $, $q\left(  \delta_{i}\right)
=\left(  \dot{t}_{\delta_{i}}^{q},\dot{m}_{\delta_{i}}^{q},\{\dot f_\mu: \mu \in\dot{J}_{\delta
_{i}}^{q}\}\right)  $ and $r\left(  \alpha_{i}\right)  =\left(  \dot{t}%
_{\delta_{i}}^{q},\dot{m}_{\delta_{i}}^{q},\{\dot f_\mu: \mu \in \dot{J}_{\delta_{i}}^{q}\cup
J_{\alpha_{i}}^{p_{Z}}\}\right)$. Since $r\leq q,$ we have that $r\leq
\overline{p},$ so $r\upharpoonright\alpha_{i}\Vdash``s_{\alpha_{i}}^{p_{Z}%
}\subseteq\dot{t}_{\delta_{i}}^{q}\mrq.$

Now, let $\alpha_{j}\in J_{\alpha_{i}}^{p_{Z}}$ (recall that the stem of
$r\left(  \alpha_{j}\right)  $ is $\dot{t}_{\delta_{j}}^{q}$). We need to
prove that $r\upharpoonright\alpha_{i}\Vdash``\dot{t}_{\delta_{i}}^{q}\cap
\dot{f}_{\alpha_{j}}\subseteq n\times n\mrq.$ Since $q$
satisfies the $K$-descending condition, we know that $q\upharpoonright
\delta_{i}\Vdash``\dot{m}_{\delta_{j}}\geq\dot{m}_{\delta_{i}}%
\mrq.$ Since $q\upharpoonright\alpha_{j}\Vdash``\dot{f}%
_{\alpha_{j}}\upharpoonright\dot{m}_{\delta_{j}}=\dot{t}_{\delta_{j}}%
^{q}\mrq,$ we get that $q\upharpoonright\delta_{i}\Vdash
``\dot{f}_{\alpha_{j}}\upharpoonright\dot{m}_{\delta_{i}}=\dot{t}_{\delta_{j}%
}^{q}\mrq.$ Since $r\leq\overline{p},$ we know that
$r\Vdash``\dot{f}_{\alpha_{i}}\cap\dot{f}_{\delta_{j}}\subseteq n\times
n\mrq.$ In particular, $\dot{t}_{\delta_{i}}^{q}$ is forced to
be disjoint with $\dot{f}_{\alpha_{j}}$ above $n,$ so we get that
$r\Vdash``\dot{t}_{\delta_{i}}^{q}\cap\dot{t}_{\delta_{j}}^{q}\subseteq
n\times n\mrq,$ hence $r\upharpoonright\alpha_{i}\Vdash``\dot
{t}_{\delta_{i}}^{q}\cap\dot{f}_{\alpha_{j}}\subseteq n\times
n\mrq,$ which is what we wanted to prove.

\begin{case}
$\alpha_{i}\in H$ and $\alpha_{i}\notin \dom\left(  q\right)  $ (so $\delta
_{i}<\alpha_{i}$).
\end{case}

In here, we have that $p_{Z}\left(  \alpha_{i}\right)  =\left(  s_{\alpha_{i}%
}^{p_{Z}},n,\{\dot f_\mu: \mu \in J_{\alpha_{i}}^{p_{Z}}\}\right)  $, $q\left(  \delta_{i}\right)
=\left(  \dot{t}_{\delta_{i}}^{q},\dot{m}_{\delta_{i}}^{q},\{\dot f_\mu: \mu \in\dot{J}_{\delta
_{i}}^{q}\}\right)  $ and $r\left(  \alpha_{i}\right)  =\left(  \dot{t}%
_{\delta_{i}}^{q},\dot{m}_{\delta_{i}}^{q},\{\dot f_\mu: \mu \in \dot{J}_{\delta_{i}}^{q}\cup
J_{\alpha_{i}}^{p_{Z}}\}\right)$. Since $r\leq q,$ we have that $r\leq\overline{p},$ so
$r\upharpoonright\delta_{i}\Vdash``s_{\alpha_{i}}^{p_{Z}}\subseteq\dot
{t}_{\delta_{i}}^{q}\mrq$ (recall that $s_{\alpha_{i}}^{p_{Z}}$
is the stem of $\overline{p}\left(  \delta_{i}\right)  $).

Now, let $\alpha_{j}\in J_{\alpha_{i}}^{p_{Z}}$ (recall that the stem of
$r\left(  \alpha_{j}\right)  $ is $\dot{t}_{\delta_{j}}^{q}$). We need to
prove that $r\upharpoonright\alpha_{i}\Vdash``\dot{t}_{\alpha_{i}}^{q}\cap
\dot{f}_{\alpha_{j}}\subseteq n\times n\mrq.$ Since $q$
satisfies the descending condition, we know that $q\upharpoonright\delta
_{i}\Vdash``\dot{m}_{\delta_{j}}\geq\dot{m}_{\delta_{i}}\mrq.$
Since $q\upharpoonright\alpha_{j}\Vdash``\dot{f}_{\alpha_{j}}\upharpoonright
\dot{m}_{\delta_{j}}=\dot{t}_{\delta_{j}}^{q}\mrq,$ we get that
$q\upharpoonright\alpha_{i}\Vdash``\dot{f}_{\alpha_{j}}\upharpoonright\dot
{m}_{\delta_{i}}=\dot{t}_{\delta_{j}}^{q}\mrq.$ Since
$r\leq\overline{p},$ we know that $r\Vdash``\dot{f}_{\alpha_{i}}\cap\dot
{f}_{\delta_{j}}\subseteq n\times n\mrq.$ In particular,
$\dot{t}_{\delta_{i}}^{q}$ is forced to be disjoint with $\dot{f}_{\alpha_{j}%
}$ above $n,$ so we get that $r\Vdash``\dot{t}_{\delta_{i}}^{q}\cap\dot
{t}_{\delta_{j}}^{q}\subseteq n\times n\mrq,$ hence
$r\upharpoonright\alpha_{i}\Vdash``\dot{t}_{\delta_{i}}^{q}\cap\dot{f}%
_{\alpha_{j}}\subseteq n\times n\mrq,$ which is what we wanted
to prove.

\begin{case}
$\alpha_{i}\in E$ and $\alpha_{i}\notin \dom\left(  q\right)  .$
\end{case}

This case is immediate by the definition.

\begin{case}
$\alpha_{i}\in E$ and $\alpha_{i}\in \dom\left(  q\right)  $ (so $\delta
_{i}=\alpha_{i}$ and $\alpha_{i}\in A$).
\end{case}

This case is immediate by the definition.

After all the cases, we can finally conclude that $r\leq q,p_{Z}.$ Since $r$
follows $K$ and $r\leq q,$ there is $r^{\prime}\in L$ such that $r^{\prime}$
and $r$ are compatible. Let $\overline{r}$ be a common extension. We now have
the following:

\begin{enumerate}
\item $\overline{r}\Vdash``\dot{X}\cap Z\neq\emptyset\mrq.$

\item $\overline{r}\Vdash``Z\subseteq\dot{A}_{\gamma}^{\xi}\mrq$
(since $\overline{r}\leq p_{Z}$).
\end{enumerate}

Hence $\overline{r}\Vdash``\dot{A}_{\gamma}^{\xi}\cap\dot{X}\nsubseteq
k\mrq,$ which is what we wanted to prove. We conclude that
$V\left[  G\right]  \models%
{\textstyle\bigcap\limits_{\gamma\in\omega_{1}}}
\mathcal{I}\left(  A^{\gamma}\right)  =[\omega]^{<\omega}.$

Let $\mathcal{B}=\left\{  f_{\alpha}\mid\alpha\in H\right\}  ,$ we now have
the following:

\begin{claim}
$\mathcal{B}$ is a \textsf{MAD} family of size $\omega_{2}.$
\end{claim}

It is easy to see that $\mathcal{B}$ is an almost disjoint family of size
$\omega_{2},$ it remains to prove that it is maximal. Let $h\in PFun\left(
\Delta\right)  $ and $A=\dom\left(  h\right)  .$ By the last claim, there is
$\gamma\in\omega_{1}$ such that $A\in\mathcal{I}\left(  \mathcal{A}^{\gamma
}\right)^{+}.$ In this way, we can find $\beta\in D_{\gamma}$ such that
$C=A\cap A_{\gamma}^{\beta}$ is infinite and $h\in V\left[  G_{\beta}\right]
,$ define $h_{1}=h\upharpoonright C$ and note that $h_{1}\in V\left[
G_{\beta+1}\right]  .$ Let $\alpha=R\left(  \beta\right)  $ (so $\beta<\alpha
$). First consider the case where $h_{1}\in\mathcal{I}\left(  \mathcal{B}%
_{\alpha}\right)  .$ In this way, there are $\alpha_{1},...,\alpha_{n}\in H$
such that $h_{1}\subseteq f_{\alpha_{1}}\cup...\cup f_{\alpha_{n}},$ so
clearly $h_{1}$ has infinite intersection with an $f_{\alpha_{i}}.$ In case
$h_{1}\in\mathcal{I}\left(  \mathcal{B}_{\alpha}\right)  ^{+},$ we will have
that \thinspace$f_{\alpha}\cap h_{1}$ is infinite by \ref{Basic E}.

Finally, we will prove the following:

\begin{claim}
$\mathfrak{ie}=\omega_{2}.$
\end{claim}

On one hand, since $\mathcal{B}$ is \textsf{MAD}, we get that $\mathfrak{ie}%
\leq\omega_{2}.$ On the other hand, since we are forcing with $\mathbb{E}%
_{\Delta}$ cofinally many times, we get that $\omega_{2}\leq\mathfrak{ie}.$ We
conclude that $\mathfrak{ie}=\omega_{2}$ holds in our model.
\end{proof}

		\bibliographystyle{plain}
		\bibliography{bibliography}
		
	\end{document}